\documentclass[a4paper]{amsart}
\usepackage[T1]{fontenc}
\usepackage[latin1]{inputenc}
\usepackage[english]{babel}
\usepackage{latexsym,amssymb,mathrsfs}
\usepackage{eucal} 
\usepackage{mathptmx}
\usepackage{dsfont} 
\usepackage{mathtools}
\usepackage{color}
\usepackage{physics}
\usepackage[shortlabels]{enumitem}
\usepackage{hyperref}
\usepackage{graphicx}
\usepackage{csquotes}
\usepackage{url}

\usepackage[style=alphabetic,natbib=true,url=false,doi=false,isbn=false,backend=bibtex]{biblatex}
\renewbibmacro{in:}{}
\usepackage{accents}
\newcommand*{\bdot}[1]{%
	\accentset{\mbox{\large .}}{#1}}

\newtheorem{theorem}{Theorem}[section]
\newtheorem*{theorem*}{Theorem}
\newtheorem{theoremx}{Theorem}
\newtheorem{lemma}[theorem]{Lemma}
\newtheorem{corollary}[theorem]{Corollary}
\newtheorem{proposition}[theorem]{Proposition}
\theoremstyle{definition}
\newtheorem{definition}[theorem]{Definition}

\theoremstyle{remark}
\newtheorem{remark}[theorem]{Remark}

\newcommand{\defin}{\vcentcolon =}
\newcommand{\C}{\mathbb{C}}
\newcommand{\R}{\mathbb{R}}

\newcommand{\I}{I}
\newcommand{\II}{I \! \! I}
\newcommand{\III}{I \! \! I \! \! I}
\newcommand{\Hyp}{\mathbb{H}}
\newcommand{\dS}{\mathrm{dS}}
\newcommand{\AdS}{\mathrm{AdS}}
\newcommand{\CFT}{\mathrm{CFT}}
\newcommand{\Teich}{\mathcal{T}}

\newcommand{\CC}{C}

\newcommand{\Vol}{\textit{V}}

\newcommand{\hyp}{\mathfrak{h}}
\newcommand{\conf}{\mathfrak{c}}
\newcommand{\id}{\textit{id}}
\newcommand{\1}{\mathds{1}}

\newcommand{\MesLam}{\mathcal{ML}}
\newcommand{\QF}{\mathcal{QF}}
\newcommand{\Ends}{\mathcal{E}}
\newcommand{\mappa}[3]{#1 \colon #2 \rightarrow #3}
\newcommand{\hsk}{\hskip0pt}

\DeclarePairedDelimiterX{\scal}[2]{\langle}{\rangle}{#1, #2}
\DeclarePairedDelimiterX{\scall}[2]{(}{)}{#1, #2}
\DeclarePairedDelimiter{\set}{\{}{\}}

\DeclareMathOperator{\arctanh}{arctanh}

\DeclareMathOperator{\Conf}{Conf}

\DeclareMathOperator{\Iso}{Iso}
\DeclareMathOperator{\Hopf}{Hopf}

\DeclareMathOperator{\Proj}{P}

\DeclareMathOperator{\SL}{SL}
\DeclareMathOperator{\ext}{ext}

\DeclareMathOperator{\SGr}{SGr}
\DeclareMathOperator{\Gr}{Gr}
\DeclareMathOperator{\Sch}{Sch}
\DeclareMathOperator{\Th}{Th}

\DeclareMathOperator{\sgn}{sgn}

\DeclareMathOperator{\divr}{div}

\title[CGC foliations and Schl\"afli formulas of hyperbolic $3$-\hsk manifolds]{Constant Gaussian curvature foliations and Schl\"afli formulas of hyperbolic 3-\hsk manifolds}

\author{Filippo Mazzoli}

\thanks{Supported by the Luxembourg National Research Fund PRIDE15/10949314/GSM/Wiese.}

\date{\today}

\address{Mathematics Research Unit,
	University of Luxembourg,
	Maison du Nombre,
	6 avenue de la Fonte,
	L-4364 Esch-sur-Alzette,
	Luxembourg}

\email{filippo.mazzoli@uni.lu}

\begin{document}
	
\begin{abstract}
	\noindent We study the geometry of the foliation by constant Gaussian curvature surfaces $(\Sigma_k)_k$ of a hyperbolic end, and how it relates to the structures of its boundary at infinity and of its pleated boundary. First, we show that the Thurston and the Schwarzian parametrizations are the limits of two families of parametrizations of the space of hyperbolic ends, defined by Labourie in \cite{labourie1992surfaces} in terms of the geometry of the leaves $\Sigma_k$. We give a new description of the renormalized volume using the constant curvature foliation. We prove a generalization of McMullen's Kleinian reciprocity theorem, which replaces the role of the Schwarzian parametrization with Labourie's parametrizations. Finally, we describe the constant curvature foliation of a hyperbolic end as the integral curve of a time-\hsk dependent Hamiltonian vector field on the cotangent space to Teichm\"uller space, in analogy to the Moncrief flow for constant mean curvature foliations in Lorenzian space-\hsk times.
\end{abstract}
	
\maketitle

\tableofcontents

\section*{Introduction}

Let $\Sigma$ be a oriented closed surface of genus larger or equal than $2$. We will denote by $\Teich_\Sigma^\conf$ the Teichm\"uller space of $\Sigma$ defined as the space of isotopy classes of conformal structures of $\Sigma$, while $\Teich_\Sigma^\hyp$ will stand for the space of isotopy classes of hyperbolic metrics of $\Sigma$. A \emph{hyperbolic end} $E$ of topological type $\Sigma \times (0,\infty)$ is a (non-\hsk complete) hyperbolic $3$-\hsk manifold homeomorphic to $\Sigma \times (0,\infty)$, whose metric completion is obtained by adding to $E$ a locally concave pleated surface homeomorphic to $\Sigma \times \set{0} \subset \Sigma \times [0,\infty)$.

It is well known by the work of Thurston that the deformation space of $(\Proj \SL_2 \C, \C \Proj^1)$-\hsk structures (also called \emph{complex projective structures}) of $\Sigma$ is in $1$-\hsk to-\hsk $1$ correspondence with the space of hyperbolic ends homeomorphic to $\Sigma \times [0,\infty)$, modulo isometry isotopic to the identity. Through this correspondence, Thurston proved that a complex projective structure $\sigma$ on $\Sigma$ is uniquely determined by a pair $(m,\mu) \in \Teich_\Sigma^\hyp$, where $m$ is the isotopy class of the metric on $\partial E$, the pleated boundary of the hyperbolic end $E$ associated to $\sigma$, and $\mu$ is the measured lamination along which $\partial E$ is bent. Moreover, every such pair $(m,\mu)$ can be realized as the data associated to some complex projective structure. This defines a homeomorphism $\Th$, which we call the \emph{Thurston parametrization}, from the space of hyperbolic ends $\Ends(\Sigma)$ to $\Teich_\Sigma^\hyp \times \MesLam_\Sigma$, where $\MesLam_\Sigma$ denotes the space of measured laminations of $\Sigma$. 

Using a purely complex analytic approach, we can also characterize the complex projective structure $\sigma$ by its induced conformal structure $c$, together with a holomorphic quadratic differential, which measures how far is the complex projective structure $\sigma$ from the Fuchsian uniformization of $c$. In this way, we obtain another map $\Sch$, which we call the \emph{Schwarzian parametrization}, from the space of hyperbolic ends $\Ends(\Sigma)$ to the bundle of holomorphic quadratic differentials over $\Teich_\Sigma^\conf$, which can be canonically identified with the cotangent bundle of the Teichm\"uller space $\Teich_\Sigma^\conf$. Similarly to $\Th$, the function $\Sch$ is an homeomorphism between $\Ends(\Sigma)$ and $T^* \Teich_\Sigma^\conf$.

By the work of \citet{labourie1991probleme}, every hyperbolic end admits a unique foliation by convex constant Gaussian curvature surfaces, with curvature $k$ varying in $(-1,0)$. For simplicity, a surface with constant Gaussian curvature equal to $k$ will be called a $k$-\hsk surface. Using the foliation by $k$-\hsk surfaces, Labourie introduced in \cite{labourie1992surfaces} two new families of parametrizations of $\Ends(\Sigma)$, indexed by $k \in (-1,0)$. As already announced by Labourie in \cite{labourie1992surfaces}, these families of maps exhibit relations with the classical Thurston and Schwarzian parametrizations. The aim of this paper is to clarify this connection, and to present a series of results that relate the $k$-\hsk surfaces of an hyperbolic end with the geometry of its conformal boundary at infinity, on one side, and of its locally concave pleated boundary, on the other.

In order to be more precise, we need to introduce some notation and to recall the properties that $k$-\hsk surfaces satisfy. Given a hyperbolic end $E$, we will denote by $\Sigma_k$ its $k$-\hsk surface, and by $\I_k$ and $\II_k$ the first and second fundamental forms of $\Sigma_k$, respectively. Since the determinant of the shape operator $B_k$ (i. e. the extrinsic curvature of $\Sigma_k$) is equal to $k + 1 > 0$ (as a consequence of the Gauss equation), we can choose the normal vector field to $\Sigma_k$ so that the second fundamental form $\II_k = \I_k(B_k \cdot, \cdot)$ is positive definite.  We define also its third fundamental form to be $\III_k \defin \I_k(B_k \cdot, B_k \cdot)$, where $B_k$ is the shape operator of $\Sigma_k$. In this way, every $k$-\hsk surface comes with the data of three Riemannian metrics $\I_k$, $\II_k$ and $\III_k$, which satisfy the following conditions:
\begin{enumerate}[i)]
	\item the Riemannian metrics $h_k \defin (- k) \I_k$ and $h_k^* \defin \left( - \frac{k}{k + 1} \right) \III_k$ are hyperbolic, i e. they have constant Gaussian curvature equal to $-1$;
	\item if $c_k$ denotes the conformal class of $\II_k$, then the identity maps $\mappa{\id}{(\Sigma_k, c_k)}{(\Sigma_k,h_k)}$ and $\mappa{\id}{(\Sigma_k, c_k)}{(\Sigma_k,h_k^*)}$ are harmonic, with opposite Hopf differentials (as observed by \citet{labourie1992surfaces}).
\end{enumerate}
Based on these remarks, we can define Labourie's parametrizations $(\hat{\Phi}_k)_k$ and $(\hat{\Psi}_k)_k$. The map $\hat{\Phi}_k$ associates to the hyperbolic end $E$ the point of the cotangent space to Teichm\"uller space given by (the isotopy class of) the conformal structure $c_k$, together with the Hopf differential of $\mappa{\id}{(\Sigma_k, c_k)}{(\Sigma_k,h_k)}$, while the map $\hat{\Psi}_k$ sends $E$ to the pair $(m_k,m_k^*) \in (\Teich_\Sigma^\hyp)^2$ of isotopy classes of $h_k$ and $h_k^*$, respectively.

For convenience here we will consider a "normalization" $\Phi_k$ of the function $\hat{\Phi}_k$, which differ from the original map simply by the multiplication by $- \frac{2 \sqrt{k + 1}}{k}$ in the fibers of $T^* \Teich_\Sigma^\conf$. To normalize the parametrization $\hat{\Psi}_k$, we can proceed following the construction of \citet{bonsante2015a_cyclic}. First we say that two hyperbolic metrics $h$ and $h^*$ are normalized if there exists a tensor $\mappa{b}{T \Sigma}{T \Sigma}$ satisfying the following properties:
\begin{enumerate}[a)]
	\item $h^*(\cdot, \cdot) = h(b \cdot, b \cdot)$;
	\item $b$ is $h$-\hsk symmetric and it has determinant $1$;
	\item $b$ is Codazzi with respect to the Levi-\hsk Civita connection $\nabla$ of $h$. In other words, for every tangent vector fields $X$ and $Y$, we have $(\nabla_X b) Y = (\nabla_Y b) X$.
\end{enumerate}
Observe that the hyperbolic metrics $h_k$ and $h_k^*$ previously defined satisfy these conditions, with $b_k \defin \frac{1}{\sqrt{k + 1}} B_k$. By a result of Schoen \cite{schoen1993the_role}, given any pair of points $m$, $m^*$ in $\Teich_\Sigma^\hyp$, there exists a normalized pair of hyperbolic metrics $h$ and $h^*$ representing $m$ and $m^*$, respectively. Then we define
\[
\begin{matrix}
j \vcentcolon & \Teich_\Sigma^\hyp \times \Teich_\Sigma^\hyp & \longrightarrow & \R \\
& (m,m^*) & \longmapsto & \int_\Sigma \tr(b) \dd{a}_h ,
\end{matrix}
\]
where $\mappa{b}{T \Sigma}{T \Sigma}$ is the operator satisfying the conditions above with respect to $h$ and $h^*$, and $\dd{a}_h$ is the area form of $h$. The function $j$ is well defined and it is symmetric in its arguments. Given any isotopy class of metrics $m^*$, the function $L_{m^*}$, which associates to each $m \in \Teich_\Sigma^\hyp$ the value $j(m,m^*)$, will be called the \emph{hyperbolic length function of $m^*$} (the reason for this name will be explained in Section \ref{subsec:psik}). Finally, we define $\Psi_k(E)$ to be the point of the cotangent space $T^* \Teich_\Sigma^\hyp$ given by $(m_k, - \frac{\sqrt{k + 1}}{k} \dd{(L_{m_k^*})})$, where $\hat{\Psi}_k(E) = (m_k, m_k^*)$. 

Our first result relates the maps $\Phi_k$ and $\Psi_k$ to the Schwarzian and Thurston parametrizations:

\begin{theoremx} \label{thm:limit_labourie_parametrizations}
	The maps $\Phi_k$ converge to the Schwarzian parametrization as $k$ goes to $0$, and the maps $\Psi_k$ converge to $\dd L \circ \Th$ as $k$ goes to $-1$, where $\dd L$ is defined as
	\[
	\begin{matrix}
	\dd L \vcentcolon & \Teich_\Sigma^\hyp \times \MesLam_\Sigma & \longrightarrow & T^* \Teich_\Sigma^\hyp \\
	& (m,\mu) & \longmapsto & (m, \dd{(L_\mu)}_m) ,
	\end{matrix}
	\]
	and $L_\mu$ denotes hyperbolic length function of the measured lamination $\mu$.
\end{theoremx}
\noindent The proof of this fact is based on the works of \citet{quinn2018asymptotically} and \citet{belraouti2017asymptotic}, which describe the limits of the geometric quantities associated to the $k$-\hsk surfaces $\Sigma_k$ as they approach the conformal boundary at infinity, and the locally concave pleated boundary, respectively.

Now, let $M$ be a convex co-\hsk compact hyperbolic $3$-\hsk manifold, and let $W_k$ and $\Vol_k^*$ be the functions
\[
W_k(M) \defin \Vol(M_k) - \frac{1}{4} \int_{\partial M_k} H_k \dd{a}_k, \qquad \Vol_k^*(M) \defin \Vol(M_k) - \frac{1}{2} \int_{\partial M_k} H_k \dd{a}_k ,
\]
where $M_k$ is the region of $M$ that is contained between the $k$-\hsk surfaces sitting in the ends of $M$. These volume functions share interesting properties with the dual volume of the convex core $V^*_\CC(M)$ and the renormalized volume $\Vol_R(M)$, respectively. First of all, they satisfy two Schl\"afli-\hsk type variation formulas that are the exact analogues of the ones of $V^*_\CC$ and $\Vol_R$, as shown by the following result:

\begin{theoremx} \label{thm:schlafli_formulas}
	The first order variations of the functions $W_k$ and $\Vol_k^*$ can be expressed as follows:
	\[
	\delta W_k = - \frac{1}{2} \dd{(\ext_{\mathcal{F}_k})}(\delta c_k) , \qquad \delta \Vol_k^* = - \frac{1}{2} \dd{(L_{\III_k})}(\delta \I_k) ,
	\]
	where $\mathcal{F}_k$ is the horizontal foliation of the quadratic differential $- \frac{\sqrt{k + 1}}{k} \Hopf(\mappa{\id}{(\Sigma, c_k)}{(\Sigma, h_k)})$, $\ext_{\mathcal{F}_k}$ is the extremal length function of $\mathcal{F}_k$, and $L_{\III_k} (\I_k) \defin - \frac{\sqrt{k + 1}}{k} L_{h_k^*} (h_k)$.
\end{theoremx}

It is not difficult to see that the functions $\Vol_k^*$ approximate the dual volume $\Vol_\CC^*$ as $k$ goes to $-1$, simply because the $k$-surfaces $\partial M_k$ converge to the boundary of the convex core of $M$. A similar property is satisfied by the $W_k$-\hsk volumes and the renormalized volume $\Vol_R$. Indeed, a simple corollary of the Theorems \ref{thm:limit_labourie_parametrizations} and \ref{thm:schlafli_formulas} is the following fact:

\begin{theoremx} \label{thm:renormalized_vol_limit_of_Wk}
	The renormalized volume of a quasi-\hsk Fuchsian manifold $M$ satisfies
	\[
	\Vol_R(M) = \lim_{k \to 0^-} \left( W_k(M) - \pi \abs{\chi(\partial M)} \arctanh \sqrt{k + 1} \right) .
	\]
\end{theoremx}

This result gives a simple description of the renormalized volume of a quasi-\hsk Fuchsian manifold $M$ in terms of the limit of the $W$-\hsk volumes of $M_k$. We underline the fact that, with this characterization, we can define the renormalized volume with a fairly simple expression in terms of the $k$-\hsk surface foliations of $M$, instead of considering the original procedure, which passes through the study of equidistant foliations of the ends of $M$.

As highlighted by the work \citet{krasnov2009symplectic}, the Schl\"afli-\hsk type variation formulas of the dual volume $\Vol_\CC^*$ and the renormalized volume $\Vol_R$ have strong implications with respect to the symplectic geometry of the spaces $T^* \Teich_\Sigma^\conf$ and $T^* \Teich_\Sigma^\hyp$, endowed with the symplectic structures $\omega^\conf$ and $\omega^\hyp$ of cotangent manifolds, respectively. Here we develop the same ideas applied to the volumes $\Vol_k^*$ and $W_k^*$, and the Labourie parametrizations $\Phi_k$ and $\Psi_k$ through the variation formulas of Theorem \ref{thm:schlafli_formulas}. In particular, we will prove:

\begin{theoremx} \label{thm:symplectic_k}
	For every $k, k' \in (-1,0)$, the function $\mappa{\Phi_k \circ \Psi_{k'}^{-1}}{(T^* \Teich_\Sigma^\hyp, \omega^\hyp)}{(T^* \Teich_\Sigma^\conf, 2 \omega^\conf)}$ is a symplectomorphism.
\end{theoremx}

We observe that this result generalizes the previous works of \citet[Theorem~1.2]{krasnov2009symplectic} and \citet[Theorem~1.11]{bonsante2015a_cyclic}, concerning the maps $\Sch \circ (\dd L \circ \Th)^{-1}$ and $\Sch \circ \Psi_k^{-1}$, respectively. Another surprisingly simple consequence of the variation formulas of the volumes $W_k$ and $\Vol_k^*$ is the following generalization of (Krasnov and Schlenker's reformulation from \cite{krasnov2009symplectic} of) McMullen's Kleinian reciprocity Theorem:

\begin{theoremx} \label{thm:reciprocities}
	Let $M$ be a $3$-\hsk manifold with boundary whose interior admits a complete convex co-\hsk compact hyperbolic structure, and denote by $\mathscr{G}(M)$ the space of isotopy classes of such structures of $M$. We set
	\[
	\phi_k \vcentcolon \mathscr{G}(M) \longrightarrow T^* \Teich_{\partial M}^\conf, \qquad \psi_k \vcentcolon \mathscr{G}(M) \longrightarrow T^* \Teich_{\partial M}^\hyp
	\] 
	to be the maps that associate, to a convex co-\hsk compact hyperbolic structure of $M$, the points of $T^* \Teich_{\partial M}^\conf$ and $T^* \Teich_{\partial M}^\hyp$ given by the vectors $(\Phi_k(E_i))_i$ and $(\Psi_k(E_i))_i$, respectively, where $E_i$ varies among the set of hyperbolic ends of $M$. Then, for every $k \in (-1,0)$, the images $\phi_k(\mathscr{G}(M))$ and $\psi_k(\mathscr{G}(M))$ are Lagrangian submanifolds of $(T^* \Teich_\Sigma^\conf, \omega^\conf)$ and $(T^* \Teich_\Sigma^\hyp, \omega^\hyp)$, respectively.
\end{theoremx}

In Section \ref{subsec:quasi_fuchsian_reciprocity} we will discuss the relations between the original McMullen's formulation of the quasi-\hsk Fuchsian reciprocity (in terms of adjoint maps) and the statement we have presented here. Theorem \ref{thm:reciprocities} generalizes \cite[Theorems~1.4,~1.5]{krasnov2009symplectic}, which state that $\Sch(\mathscr{G}(M))$ and $(\dd L \circ \Th)(\mathscr{G}(M))$ are Lagrangian submanifolds of $T^* \Teich_{\partial M}^\conf$ and $T^* \Teich_{\partial M}^\hyp$, respectively.

As last (but not least) application of the tools developed here, we prove that the $k$-\hsk surface foliations of hyperbolic ends correspond to integral curves of $k$-\hsk dependent Hamiltonian vector fields on the spaces $T^* \Teich_\Sigma^\conf$ and $T^* \Teich_\Sigma^\hyp$. This phenomenon can be interpreted as the analogous of what observed by \citet{moncrief1989reduction} for constant mean curvature foliations in $3$-\hsk dimensional Lorenzian space-\hsk times. If $\bdot{\Phi}_k$ and $\bdot{\Psi}_k$ denote the vector fields $\dv{k} \Phi_k$ and $\dv{k} \Psi_k$, respectively, then we will prove:

\begin{theoremx} \label{thm:kflow_hamiltonian}
	The $k$-\hsk dependent vector field $\bdot{\Phi}_k \circ \Phi_k^{-1}$ (resp. $\bdot{\Psi}_k \circ \Psi_k^{-1}$) is Hamiltonian with respect to the real cotangent symplectic structure of $T^* \Teich_\Sigma^\conf$ (resp. $T^* \Teich_\Sigma^\hyp$), with Hamiltonian function $- \frac{1}{8(k + 1)} m_k \circ \Phi_k^{-1}$ (resp. $- \frac{1}{2k} m_k \circ \Psi_k^{-1}$), where $\mappa{m_k}{\Ends(\Sigma)}{\R}$ sends the hyperbolic end $E$ into the integral of the mean curvature $\int_{\Sigma_k} H_k \dd{a}_k$ of its $k$-\hsk surface $\Sigma_k$.
\end{theoremx}

We observe that the role of the area functional in \cite{moncrief1989reduction} as Hamiltonian function here is replaced by the integral of the mean curvature, which coincides with the hyperbolic length $L_{\III_k} (\I_k)$ considered above.

As a final remark, we summarize the transitional properties that $k$-\hsk surfaces possess in relation to the boundary of the convex core and the conformal boundary at infinity of convex co-\hsk compact hyperbolic manifolds in the following table:
\begin{table}[ht]
	\centering
	\def\arraystretch{1.1}
	\begin{tabular}{|c|c|c|}
		\hline
		On $\partial \CC M$ & On $\partial M_k$ & On $\partial_\infty M$ \\
		\hline\hline
		& Conformal class $c_k = [\II_k]$ & Conformal structure $c$ \\
		\hline 
		Induced metric $m$ & First fund. form $\I_k$ & \\
		\hline
		& Measured foliation $\mathcal{F}_k$ & Measured foliation $\mathcal{F}$ \\
		\hline
		Bending measure $\mu$ & Third fund. form $\III_k$ & \\
		\hline
		& Extremal length $\ext_{\mathcal{F}_k}(c_k)$ & Extremal length $\ext_\mathcal{F}(c)$ \\
		\hline
		Hyperbolic length $L_\mu(m)$ & $L_{\III_k}(\I_k) = \int_{\partial M_k} H_k \dd{a}_k$  & \\
		\hline
		& Param. $\Phi_k$ (Cor \ref{cor:convergence_phik}) & Schwarzian param. $\Sch$ \\
		\hline
		Thurston param. $\Th$ & Param. $\hat{\Psi}_k$ (Thm \ref{thm:convergence_psi_k_hat}) & \\
		\hline
		& Volume $W_k$ & Renorm. volume $\Vol_R$ \\
		\hline
		Dual volume $V^*_\CC$ & Volume $\Vol_k^*$ & \\
		\hline
		& Thm \ref{thm:schlafli_Wk} & \cite[Thm~1.2]{schlenker2017notes}  \\
		& $\delta W_k = \frac{1}{2} \dd{(\ext_{\mathcal{F}_k})(\delta c_k)}$ & $\delta \Vol_R = -\frac{1}{2} \dd{(\ext_\mathcal{F})}(\delta c)$ \\
		\hline
		\cite[Lemma~2.2]{krasnov2009symplectic}, \cite{mazzoli2018the_dual} & Thm \ref{thm:schlafli_Vk*} & \\
		$\delta \Vol_\CC^* = -\frac{1}{2} \dd{(L_\mu)}(\delta m)$ & $\delta \Vol_k^* = - \frac{1}{2} \dd{(L_{\III_k})(\delta \I_k)}$ & \\
		\hline
		& $\phi_k(\mathscr{G}(M))$ is & McMullen's Kleinian \\
		& Lagrangian (Thm \ref{thm:reciprocities}) & reciprocity \cite{mcmullen1998complex} \\
		\hline
		$(\dd L \circ \Th)(\mathscr{G}(M))$ is & $\psi_k(\mathscr{G}(M))$ is & \\
		Lagrangian \cite[Thm~1.4]{krasnov2009symplectic} & Lagrangian (Thm \ref{thm:reciprocities}) & \\
		\hline
	\end{tabular}
\end{table}

\subsection*{Outline of the paper} 
	
In the first Section we recall the necessary background about harmonic and minimal Lagrangian maps between surfaces, the properties of $k$-\hsk surfaces and classical parametrizations of the space of hyperbolic ends, namely the Schwarzian and Thurston parametrizations. 

The second Section is dedicated to the definition of the Labourie parametrizations $\Phi_k$ and $\Psi_k$, and to the proof of Theorem \ref{thm:limit_labourie_parametrizations}, which is divided in two parts, Corollary \ref{cor:convergence_phik} and Theorem \ref{thm:convergence_psi_k_hat}.

Section 3 focuses on the Schl\"afli formulas of Theorem \ref{thm:schlafli_formulas} (see Theorems \ref{thm:schlafli_Wk} and \ref{thm:schlafli_Vk*}). While the proof of the variation of formula of $\Vol_k^*$ is essentially a combination of results extracted from the works of \citet{bonsante2013a_cyclic}, \cite{bonsante2015a_cyclic}, the variation formula of the volumes $W_k$ will require a bit more care. The main technical ingredients of our analysis will be a new way to express the variation of the $W$-\hsk volume (see Proposition \ref{prop:variation_W_volume} in the Appendix) and Gardiner's formula \cite[Theorem~8]{gardiner1984measured} for the differential of the extremal length function.  In Section \ref{subsec:renormalized_volume} we will combine the results from Section 2 with Theorem \ref{thm:schlafli_formulas} and with the Schl\"afli formula of the renormalized volume $\Vol_R$ (from \cite[Theorem~1.2]{schlenker2017notes}) to deduce a new description of the renormalized volume of a convex co-\hsk compact hyperbolic $3$-\hsk manifold in terms of the limit of the volumes $W_k$ (Theorem \ref{thm:renormalized_vol_limit_of_Wk}).

The remaining sections are dedicated to the proofs of Theorems \ref{thm:symplectic_k}, \ref{thm:reciprocities} and \ref{thm:kflow_hamiltonian}. As we will see, these results will follow as fairly elementary applications of what described in Sections 2 and 3. In Section 4 we use two "relative versions" $w_k$ and $v_k^*$ of the volumes $W_k$ and $\Vol_k^*$ to describe the pullback of the Liouville forms of $T^* \Teich_\Sigma^\conf$ and $T^* \Teich_\Sigma^\hyp$ (Lemma \ref{lem:pullback_liouville_forms}) under the Labourie parametrizations $\Phi_k$ and $\Psi_k$, respectively. As immediate consequence, we will deduce Theorem \ref{thm:symplectic_k}. In Section 5 we will prove Theorem \ref{thm:reciprocities}, whose proof is a simple application of the Schl\"afli formulas of Theorem \ref{thm:schlafli_formulas} and of the dual Bonahon-\hsk Schl\"afli formula (see \cite{krasnov2009symplectic}, \cite{mazzoli2018the_dual}). Finally, Section 6 is devoted to the proof of Theorem \ref{thm:kflow_hamiltonian}, which is based on the formulas of Lemma \ref{lem:pullback_liouville_forms} and on a elementary application of Cartan formula (see Lemma \ref{lem:deriv_pullback_liouville_time_dep}).

\subsection*{Acknowledgments}
	
I would like to thank my advisor Jean-\hsk Marc Schlenker for his help and support, and Keaton Quinn, for interesting discussions during his visit in the University of Luxembourg that helped me to develop this work.

\section{Preliminaries}

In our presentation, $\Sigma$ will always denote a closed orientable surface of genus $g \geq 2$.

\begin{definition}
	Let $\Sigma$ be a surface. Two Riemannian metrics $g$, $g'$ on $\Sigma$ are \emph{conformally equivalent} if there exists a smooth function $\alpha \in \mathscr{C}^\infty(\Sigma)$ such that $g' = e^{2 \alpha} g$. A \emph{conformal structure} $c$ on $\Sigma$ is an equivalence class of Riemannian metrics with respect to the relation above, together with a choice of a orientation of $\Sigma$. A \emph{hyperbolic metric} $h$ on $\Sigma$ is a Riemannian metric with Gaussian curvature constantly equal to $- 1$.
\end{definition}

Given any surface $\Sigma$, we will denote by $\Teich_\Sigma$ the \emph{Teichm\"uller space} of $\Sigma$. By the uniformization theorem, $\Teich_\Sigma$ can be interpreted either as the space of isotopy classes of hyperbolic metrics (complete Riemannian metrics of constant curvature $-1$), or of conformal structures on $\Sigma$. We will write $\Teich^\hyp_\Sigma$ ($\hyp$ for \emph{hyperbolic}) when we want to emphasize the first interpretation, and $\Teich^\conf_\Sigma$ ($\conf$ for \emph{conformal}) in latter case.

Given a conformal structure $c$ on $\Sigma$, we denote by $Q(\Sigma,c)$ the space of holomorphic quadratic differentials of $(\Sigma,c)$. By the Riemann-\hsk Roch theorem, $Q(\Sigma,c)$ is a vector space of complex dimension $3 g - 3$. It is well known that the cotangent space to Teichm\"uller space $\Teich_\Sigma^\conf$ at the isotopy class of $c$ can be naturally identified with the vector space $Q(\Sigma,c)$.

\subsection{Harmonic and minimal Lagrangian maps}

In what follows, we briefly recall the definitions of harmonic and minimal Lagrangian maps between hyperbolic surfaces, and the relative results that we will need in our presentation.

\begin{definition} \label{def:harmonic}
	Let $c$ and $g$ be a conformal structure and a Riemannian metric on $\Sigma$, respectively. A smooth map $\mappa{u}{(\Sigma,c)}{(\Sigma,g)}$ is \emph{harmonic} if the $(2,0)$-\hsk part of $u^* g$ with respect to the conformal structure $c$ is a holomorphic quadratic differential. In such case, we call $(u^* g)^{(2,0)}$ the \emph{Hopf differential} of $u$.
	
	Equivalently, $u$ is harmonic if there exists a (and, consequently, for any) Riemannian metric $g'$ in the conformal class $c$, such that the $g'$-\hsk traceless part of $u^* g$ is a $g'$-\hsk divergence free tensor (see \cite[p.~45-46]{tromba2012teichmuller} for the equivalence of these definitions).
\end{definition}

\begin{remark} \label{rmk:rels_hopf_traceless_part}
	If $\mappa{f}{(\Sigma,[g'])}{(\Sigma,g)}$ is harmonic with Hopf differential $q$, then the $g'$-\hsk traceless part of $f^* g$ is equal to $2 \Re q$ ($[g']$ is the conformal class of $g'$).
\end{remark}

\begin{theorem}[{See e. g. \cite{sampson1978some_properties}}]
	Let $c$ be a conformal structure on $\Sigma$. Then, for any hyperbolic metric $h$ on $\Sigma$, there exists a unique holomorphic quadratic differential $q(c,h) \in Q(\Sigma,c)$, and a unique diffeomorphism $\mappa{u(c,h)}{(\Sigma,c)}{(\Sigma,h)}$ isotopic to the identity, such that $u(c,h)$ is harmonic with Hopf differential $q(c,h)$.
\end{theorem}

\begin{theorem}[{\cite[Theorem~3.1]{wolf1989the_teichmuller}}] \label{thm:wolf_param}
	For every $c \in \Teich_\Sigma^\conf$, the function
	\[
	\begin{matrix}
	\varphi_c \vcentcolon & \Teich_\Sigma^\hyp & \longrightarrow & Q(\Sigma,c) \\
	& [h] & \longmapsto & q(c,h) ,
	\end{matrix}
	\]
	is a diffeomorphism.
\end{theorem}

\begin{definition}[{See \cite[Proposition~1.3]{bonsante2013a_cyclic}}] \label{def:minimal_lagr}
	Let $h$ and $h'$ be two hyperbolic metrics on $\Sigma$. A diffeomorphism $\mappa{f}{(\Sigma,h)}{(\Sigma,h')}$ is \emph{minimal Lagrangian} if it is area-\hsk preserving, and its graph is a minimal surface inside $(\Sigma^2, h \oplus h')$.
	
	Equivalently, $\mappa{f}{(\Sigma,h)}{(\Sigma,h')}$ is minimal Lagrangian if there exists a conformal structure $c$ on $\Sigma$ such that $f = u(c,h') \circ u(c,h)^{-1}$ and $q(c,h') = - q(c,h)$, with the notation introduced in Definition \ref{def:harmonic}.
\end{definition}

\begin{remark}
	Using the first description of minimal Lagrangian maps, the conformal structure $c$, appearing in the second definition, can be recovered as the conformal class of the induced metric on the graph of $f$ from the metric $h \oplus h'$ (by identifying the graph of $f$ with $\Sigma$ using one of the projections onto $\Sigma$). Moreover, the projections of the graph of $f$ onto $(\Sigma,h)$ and $(\Sigma,h')$ are harmonic with respect to $c$.
\end{remark}

\begin{theorem}[{\cite{labourie1992surfaces}, \cite{schoen1993the_role}}] \label{thm:labourie_minimal_lagr}
	For every hyperbolic metric $h$ and for every isotopy class $m' \in \Teich_\Sigma^\hyp$, there exists a unique hyperbolic metric $h' \in m'$ and a unique operator $\mappa{b}{T \Sigma}{T \Sigma}$ such that:
	\begin{enumerate}[i)]
		\item $h' = h (b \cdot, b \cdot)$;
		\item $b$ is $h$-\hsk self-\hsk adjoint and positive definite,
		\item $\det b = 1$;
		\item $b$ is Codazzi with respect to the Levi-\hsk Civita connection $\nabla$ of $h$, i. e. $(\nabla_X b)Y = (\nabla_Y b)X$ for every $X$ and $Y$.	
	\end{enumerate}
\end{theorem}

\begin{definition} \label{def:labourie_operator}
	Whenever we have a pair of hyperbolic metrics $h$, $h'$ and an operator $b$ as in the statement above, we say that the pair $h$, $h'$ is \emph{normalized}, and that $b$ is the \emph{Labourie operator} of the couple $h$, $h'$.
\end{definition}

It turns out that, if $h$ and $h'$ are a normalized pair of hyperbolic metrics with Labourie operator $b$, then the conformal class $c$ of the Riemannian metric $h(b \cdot, \cdot)$ is such that the maps
\[
(\Sigma,h) \stackrel{\id}{\longleftarrow} (\Sigma,c) \stackrel{\id}{\longrightarrow} (\Sigma,h')
\]
are harmonic, with opposite Hopf differentials. Therefore, Theorem \ref{thm:labourie_minimal_lagr} can be reformulated in the following way:

\begin{theorem}[{\cite{labourie1992surfaces}, \cite{schoen1993the_role}}] \label{thm:labourie_minimal_lagr1}
	The function
	\[
	\begin{matrix}
	\mathcal{H} \vcentcolon & T^* \Teich_\Sigma^\conf & \longrightarrow & \Teich_\Sigma^\hyp \times \Teich_\Sigma^\hyp \\
	& (c,q) & \longmapsto & (\varphi_c^{-1}(q), \varphi_c^{-1}(-q)) ,
	\end{matrix}
	\]
	is a diffeomorphism (here $\varphi_c$ denotes the harmonic parametrization of Theorem \ref{thm:wolf_param}).
\end{theorem}

\subsection{Constant extrinsic curvature surfaces} \label{subsec:constant_extr_curvature}

Let $\Sigma$ be a (space-\hsk like) surface immersed in a Riemannian (Lorentzian) $3$-\hsk manifold $M$ of constant sectional curvature $\sec(M)$, with first and second fundamental forms $\I$ and $\II$, and shape operator $B$. We denote by $K_e$ its \emph{extrinsic curvature}, i. e. $K_e = \det B$, and by $K_i$ its \emph{intrinsic curvature}, i. e. the Gauss curvature of the Riemannian metric $\I$. For convenience, we define $\sgn(M)$ to be $+ 1$ if $M$ is a Riemannian manifold, and $-1$ is $M$ is Lorentzian. Then, the Gauss-\hsk Codazzi equations of $(\Sigma,\I,\II)$ can be expressed as follows:
\begin{gather}
K_i =  \sgn(M) \ K_e + \sec(M) , \nonumber \\
(\nabla_U B)V = (\nabla_V B)U \quad \forall U, V , \label{eq:codazzi_eq}
\end{gather}
where $U$ and $V$ are tangent vector fields to $\Sigma$, and $\nabla$ is the Levi-Civita connection of the metric $\I$. The \emph{third fundamental form} of $\Sigma$ is the symmetric $2$-\hsk tensor $\I(B \cdot, B \cdot)$.

\begin{definition}
	Let $\Sigma$ be an immersed (space-\hsk like) surface of a Riemannian (Lorentzian) $3$-\hsk manifold $M$. We say that $\Sigma$ is \emph{strictly convex} if its second fundamental form $\II$ is positive definite.
\end{definition}

\begin{remark}
	Observe that the notion of strict convexity implicitly depends on the choice of a normal vector field of $\Sigma$. Moreover, if $\Sigma$ is a strictly convex surface, then its third fundamental form is a Riemannian metric too.
\end{remark}

Let $\Sigma$ be a surface immersed in a hyperbolic $3$-\hsk manifold $M$. The Gauss equation in this case has the following form:
\begin{gather}
K_i = K_e - 1 . \label{eq:gauss_eq}
\end{gather}
Given $k \in (-1, 0)$, we say that $\Sigma$ is a \emph{$k$-\hsk surface} of $M$ if its intrinsic curvature is constantly equal to $k$. If we define the shape operator of $\Sigma$ using the normal vector field of $\Sigma$ that points to the convex side of $\Sigma$, then the second fundamental form $\II$ of $\Sigma$ has strictly positive principal curvatures, since $\det B = K_e = k + 1 > 0$. Therefore $\II$ is a positive definite symmetric bilinear form; in other words, $\Sigma$ is strictly convex.

In order to give a geometric interpretation of the third fundamental form $\III$ of $\Sigma$, we need to clarify the connection between the hyperbolic $3$-\hsk space $\Hyp^3$ and the \emph{de Sitter} $3$-\hsk space $\dS^3$, which is a Lorentzian analogue of the $3$-\hsk dimensional unit sphere $\mathbb{S}^3 \subset \R^4$. Let $\R^{3,1}$ denote the $4$-\hsk dimensional Minkowski space, i. e. the vector space $\R^4$ endowed with a Lorentzian scalar product of signature $(3,1)$. Then the hyperbolic space $\Hyp^3$ can be viewed as (a connected component of) the set of vectors $x \in \R^{3,1}$ satisfying $\scal{x}{x} = - 1$. Similarly, the de Sitter space is defined as the set of vectors $y$ satisfying $\scal{y}{y} = 1$. The projection of $\R^4$ onto the projective space $\R \Proj^3$ sends the light cone into the quadric $\Proj \set{x \mid \scal{x}{x} = 0}$ of $\R \Proj^3$. The polarity correspondence determined by this quadric allows to construct, starting from a strictly convex surface $\widetilde{\Sigma}$ in $\Hyp^3$, an associated strictly convex surface $\widetilde{\Sigma}^*$ in $\dS^3$, whose points are the polar-\hsk duals of the tangent spaces to the surface $\widetilde{\Sigma}$ in $\Hyp^3$. Moreover, if $\widetilde{\Sigma}$ is the lift to $\Hyp^3$ of a surface $\Sigma$ sitting inside some hyperbolic end $E$ (see Definition \ref{def:hyp_end}), then $\widetilde{\Sigma}$ determines a space-\hsk like surface $\Sigma^*$ inside a \emph{maximal global hyperbolic spatially compact de Sitter spacetime} $E^*$ (see e. g. \cite{mess2007lorentz} for details). Finally, in this description, the first fundamental form $\I^*$ of $\Sigma^*$ coincides with the tensor $\III$, and the second fundamental forms $\II^*$ and $\II$ are essentially the same ($\II = \pm \II^*$, depending on the conventions). We refer to \cite{schlenker2002hypersurfaces} for more detailed description of this correspondence.

Now, the Gauss equation of the dual surface $(\Sigma^*,\I^*,\II^*)$ is
\begin{equation}
K_i^* = - K_e^* + 1 . \label{eq:dual_gauss_eq}
\end{equation}
Since the shape operator $B^*$ of $\Sigma^*$ coincides with $\pm B^{-1}$, the surface $(\Sigma,\I,\II)$ has extrinsic curvature $K_e$ if and only if $(\Sigma^*, \I^*, \II^*)$ has extrinsic curvature $K_e^* = K_e^{-1}$. Combining this fact with the Gauss equations \eqref{eq:gauss_eq} and \eqref{eq:dual_gauss_eq}, we see that, if $\Sigma$ is a $k$-\hsk surface, then the tensors
\[
h \defin - k \ \I  \quad \text{and} \quad h^* \defin - \frac{k}{k + 1} \ \III
\]
are Riemannian metrics of constant curvature $-1$. We also set $c$ to be the conformal class of $\II$.

\begin{lemma}
	Let $\Sigma$ be a strictly convex surface immersed in a Riemannian (or Lorentzian) $3$-\hsk manifold $M$. The following are equivalent:
	\begin{itemize}
		\item the surface $\Sigma$ has constant extrinsic curvature;
		\item the identity map $\mappa{\id}{(\Sigma,c)}{(\Sigma,\I)}$ is harmonic.
	\end{itemize}
	\begin{proof}
		The Levi-Civita connection $\nabla^{\II}$ of the Riemannian metric $\II$ satisfies
		\[
		\nabla^{\II}_U V = \nabla_U V + \frac{1}{2} B^{-1} (\nabla_U B) V ,
		\]
		where $U$ and $V$ are tangent vector fields to $\Sigma$, and $\nabla$ is the Levi-Civita connection of $\I$. This relation can be easily proved by showing that the right-\hsk hand side, as a function of $U$ and $V$, defines a connection which is torsion-\hsk free and compatible with $\II$.
		The first property follows from the fact that $\nabla$ is torsion-\hsk free, and from the Codazzi equation \eqref{eq:codazzi_eq} satisfied by $B$. The compatibility with respect to $\II$ can be derived by the compatibility of $\nabla$ with respect to $\I$, and by the fact that $B$ is $\I$-\hsk self-\hsk adjoint. 
		
		The map $\mappa{\id}{(\Sigma,\II)}{(\Sigma,\I)}$ is harmonic if and only if the $\II$-\hsk traceless part of $\I$ is the real part of a holomorphic quadratic differential. By the results of \cite[Chapter~2]{tromba2012teichmuller}, this is equivalent to saying that $\I - \frac{H}{2 K_e} \II$ is divergence free with respect to $\nabla^{\II}$ and traceless with respect to $\II$ (which is true by definition).
		Using the expression of $\nabla^{\II}$ above, we can prove that
		\[
		\divr_{\II}\left(\I - \frac{H}{2 K_e} \II \right) = - \frac{1}{2} \dd(\ln K_e) .
		\]
		From this equation the statement is clear.
	\end{proof}
\end{lemma}

This fact proves that both the maps
\[
(\Sigma,h) \stackrel{\id}{\longleftarrow} (\Sigma,c) \stackrel{\id}{\longrightarrow} (\Sigma,h^*)
\]
are harmonic. If $\hat{T}$ denotes the $\II$-\hsk traceless part of the symmetric tensor $T$, then we have
\[
\hat{h} = - k \left( \I - \frac{H}{k + 1} \II \right), \qquad \hat{h}^* = - \frac{k}{k + 1} \left( \III - H \II \right) .
\]
Using the relation $B^2 - H B + K_e \1 = 0$, we see that $\hat{h} = - \hat{h}^*$. This shows that the two identity maps above have opposite Hopf differentials or, equivalently, that the map $\mappa{\id}{(\Sigma,h)}{(\Sigma,h^*)}$ is minimal Lagrangian (see Definition \ref{def:minimal_lagr}).

\subsection{The space of hyperbolic ends}

\begin{definition} \label{def:hyp_end}
	Given $\Sigma$ a closed surface, a \emph{hyperbolic end} $E$ of topological type $\Sigma \times [0, \infty)$ is a hyperbolic $3$-\hsk manifold with underlying topological space $\Sigma \times (0,\infty)$ and whose metric completion $\overline{E} \cong \Sigma \times [0,\infty)$ is obtained by adding to $E$ a locally concave pleated surface $\Sigma \times \set{0} \subset \Sigma \times [0,\infty)$. We will denote by $\partial E$ the locally concave pleated boundary of $E$.
	
	Two hyperbolic ends $E = (\Sigma \times (0,\infty), g)$ and $E' = (\Sigma \times (0,\infty), g')$ are equivalent if there exists an isometry between them that is isotopic to $\id_{\Sigma \times (0,\infty)}$. We set $\Ends(\Sigma)$ to be the space of equivalence classes of hyperbolic ends of topological type $\Sigma \times (0,\infty)$.
\end{definition}

Let $E$ be a hyperbolic end. The manifold $\overline{E} \cong \Sigma \times [0,\infty)$ can be compactified by adding a topological surface "at infinity" $\partial_\infty E \defin \Sigma \times \set{\infty}$. The $(\Iso^+(\Hyp^3),\Hyp^3)$-\hsk structure on $E$ naturally determines a $(\Proj \SL_2 \C, \C \Proj^1)$-\hsk structure (also called \emph{complex projective structure}) $\sigma_E$ on $\partial_\infty E$, coming from the action of $\Iso^+(\Hyp^3) \cong \Proj \SL_2 \C$ on the boundary at infinity $\partial_\infty \Hyp^3 \cong \C \Proj^1$. 

By a classical construction due to Thurston, it is possible to invert this process: given a complex projective structure $\sigma$ on a surface $\Sigma$, there exists a hyperbolic end $E$ of topological type $\Sigma \times (0,\infty)$ whose induced complex projective structure on $\partial_\infty E$ coincides with $\sigma$. The universal cover $\widetilde{E}$ of $E$ can be locally described as the envelope of those half-\hsk spaces $H$ of $\Hyp^3$ satisfying $\overline{H} \cap \partial_\infty \Hyp^3 = D$, where $D$ varies over the developed maximal discs of $(\widetilde{\Sigma}, \tilde{\sigma})$ in $\partial_\infty \Hyp^3 = \C \Proj^1$. This construction establishes a one-\hsk to-\hsk one correspondence between the space of hyperbolic ends $\Ends(\Sigma)$ and the deformation space of complex projective structures on $\Sigma$. We refer to \cite{kamishima1992deformation} for a more detailed exposition of Thurston's construction.

\subsubsection*{The Schwarzian parametrization}

Let $E$ be a hyperbolic end. Following the notation introduced above, we denote by $c_0$ the underlying conformal structure of $\sigma_E$, and by $\sigma_0$ the "Fuchsian structure" of $c_0$, i. e. the complex projective structure on $\Sigma = \partial_\infty E$ determined by the uniformization map of $(\widetilde{\Sigma},\tilde{c}_0)$. The set of complex projective structures with underlying conformal structure $c_0$ can be interpreted as an affine space over the space of holomorphic quadratic differentials of $(\Sigma,c_0)$, and the correspondence sends each element $\sigma - \sigma_0$ into the \emph{Schwarzian derivative} of $\sigma$ with respect to $\sigma_0$ (see \cite{dumas2009complex} for details). In particular, the element $\sigma_E - \sigma_0$ determines a unique holomorphic quadratic differential $q_0$ of $(\Sigma,c_0)$, called the \emph{Schwarzian at infinity} of $E$. The resulting map
\[
\begin{matrix}
\Sch \vcentcolon & \Ends(\Sigma) & \longrightarrow & T^* \Teich_\Sigma^\conf \\
& [E] & \longmapsto & (c_0,q_0) ,
\end{matrix}
\]
gives a parametrization of the space of hyperbolic ends $\Ends(\Sigma)$, which we will call the \emph{Schwarzian parametrization}.

\subsubsection*{The Thurston parametrization}

The Schwarzian parametrization of the space of hyperbolic ends $\Ends(\Sigma)$ uses the geometric structure of the boundary at infinity $\partial_\infty E$ of $E$. In the following we will describe a analogous construction, due to Thurston, involving the shape of the convex pleated boundary $\partial E$, instead of $\partial_\infty E$.

The surface $\partial E$ is a topologically embedded surface in $E$, which is almost everywhere totally geodesic. The set of points where $\partial E$ is not locally shaped as an open set of $\Hyp^2$ is a closed subset $\lambda$ that is disjoint union of simple (not necessarily closed) complete geodesics. The path metric of $\partial E$ is an actual hyperbolic metric $h \in \Teich_\Sigma^\hyp$, and the structure of the singular locus $\lambda$ can be described using the notion of \emph{tranverse measured lamination}. In the simple case of $\lambda$ composed by disjoint simple closed geodesics, each leaf $\gamma_i$ of $\lambda$ has an associated exterior dihedral angle $\vartheta_i$, which measures the bending between the totally geodesic portions of $\partial E$ meeting along $\gamma_i$. Given any geodesic arc $\alpha$ transverse to $\lambda$, we can define the transverse measure $\mu \defin \sum_i \vartheta_i \ \gamma_i$ along a geodesic segment $\alpha$ to be the sum $\sum_i \vartheta_i \ i(\gamma_i,\alpha)$, where $i(\gamma_i,\alpha)$ is the geometric intersection between $\alpha$ and $\gamma_i$. Using an approximation procedure, we can generalize the construction above to a generic support $\lambda$, obtaining a measured lamination $\mu \in \MesLam_\Sigma$, which measures the amount of bending that occurs transversely to $\lambda$. The datum of the hyperbolic metric $h$ and the measured lamination $\mu$ is actually sufficient to describe the entire hyperbolic end. In other words, the map
\[
\begin{matrix}
\Th \vcentcolon & \Ends(\Sigma) & \longrightarrow & \Teich^\hyp_\Sigma \times \MesLam_\Sigma \\
& [E] & \longmapsto & (h, \mu)
\end{matrix}
\]
parametrizes the space of hyperbolic ends (for a detailed proof of this result, see \cite[Section~2]{kamishima1992deformation}). We will call $\Th$ the \emph{Thurston parametrization} of $\Ends(\Sigma)$.

\section{Foliations by \texorpdfstring{$k$}{k}-\hsk surfaces}

This Section is mainly devoted to the description of two families of parametrizations of the space of hyperbolic ends $\Ends(\Sigma)$, denoted by $(\Phi_k)_k$ and $(\Psi_k)_k$, firstly introduced by \citet{labourie1992surfaces}, and further investigated by Bonsante, Mondello and Schlenker in \cite{bonsante2013a_cyclic} and \cite{bonsante2015a_cyclic}. After having recalled the necessary background, we will establish a connection between the asymptotic of these maps and the classical Schwarzian and Thurston parametrizations, applying the recent works of \citet{quinn2018asymptotically} and \citet{belraouti2017asymptotic}, respectively.

\begin{theorem}[{\cite[Théorème~2]{labourie1991probleme}}] \label{thm:existence_k_foliation}
	Every hyperbolic end has a unique foliation by $k$-surfaces $\Sigma_k$, with $k$ varying in $(-1,0)$. As $k$ goes to $-1$, the $k$-surface $\Sigma_k$ approaches the concave pleated boundary of $E$, and as $k$ goes to $0$, $\Sigma_k$ approaches the conformal boundary at infinity of $E$.
\end{theorem}

Before describing the maps $(\Phi_k)_k$ and $(\Psi_k)_k$, we need to introduce some notation that we will useful (and used) in the rest of the paper. Given $E$ a hyperbolic end, with $k$-\hsk surface foliation $(\Sigma_k)_k$, we let $\I_k$, $\II_k$ and $\III_k$ denote the first, second and third fundamental forms of $\Sigma_k$. Moreover, we set $h_k$ and $h_k^*$ to be the hyperbolic metrics $- k \ \I_k$ and $- \frac{k}{k + 1} \ \III_k$, respectively, and $c_k$ to be the conformal class of $\II_k$. Finally, we will denote by $q_k$ the holomorphic quadratic differential
\[
- \frac{2 \sqrt{k + 1}}{k} \Hopf((\Sigma_k,c_k) \rightarrow (\Sigma_k,h_k)) = \frac{2 \sqrt{k + 1}}{k} \Hopf((\Sigma_k,c_k) \rightarrow (\Sigma_k,h_k^*)) .
\]
The choice of the multiplicative constant in the definition of $q_k$ may look arbitrary at this point of the exposition, but it will be crucial in the following (see for instance Corollary \ref{cor:convergence_phik} and Remark \ref{rmk:multipl_constant}). The holomorphic quadratic differential $q_k$ satisfies
\begin{equation} \label{eq:hopf_diff_I_and_III}
2 \Re q_k  = 2 \sqrt{k + 1} \left( \I_k - \frac{H_k}{2(k + 1)} \II_k \right) = - \frac{2}{\sqrt{k + 1}} \left( \III_k - \frac{H_k}{2} \II_k \right) .
\end{equation}
For future references, we also observe that the area forms with respect to $\I_k$ and $\II_k$ differ by a multiplicative constant, as follows:
\begin{equation} \label{eq:area_forms}
\dd{a}_{\I_k} = \frac{1}{\sqrt{\det B_k}} \dd{a}_{\II_k} = \frac{1}{\sqrt{k + 1}} \dd{a}_{\II_k} .
\end{equation}

\subsection{The parametrizations \texorpdfstring{$\Phi_k$}{Phik}} \label{subsec:phik}

The first class of parametrizations described by \citet{labourie1992surfaces} is given by the following maps: for every $k \in (-1,0)$ we define the function
\[
\begin{matrix}
\Phi_k \vcentcolon & \Ends(\Sigma) & \longrightarrow & T^* \Teich_\Sigma^\conf \\
& [E] & \longmapsto & (c_k,q_k) ,
\end{matrix}
\]
which associates, to every hyperbolic end $E$, the point of the cotangent space to Teichm\"uller space $(c_k,q_k)$ determined by the unique $k$-surface $\Sigma_k$ contained in $E$, as above. We have:

\begin{theorem}[{\cite[Théorème~3.1]{labourie1992surfaces}}] \label{thm:labourie_parametrization}
	The function $\Phi_k$ is a diffeomorphism for every $k \in (-1,0)$.
\end{theorem}

In the following we will see how the maps $\Phi_k$ relate to the Schwarzian parametrization $\Sch$. Using the hyperbolic Gauss map (see e. g. \cite{labourie1991probleme}), we can think about the families $(\I_k)_k$, $(\II_k)_k$ and $(\III_k)_k$ as paths in the space of $(2,0)$-\hsk symmetric tensors over the surface $\partial_\infty E$, which does not depend on $k$. In this way we can study the asymptotic of these geometric quantities as $k$ goes to $0$.

In a recent work \cite{quinn2018asymptotically}, Quinn introduced the notion of \emph{asymptotically Poincaré families of surfaces} inside a hyperbolic end $E$, and he determined a connection between their geometric properties and the complex projective structure at infinity of $E$. The foliation by $k$-surfaces is an example of such families and the asymptotic of their fundamental forms is understood. In order to do not introduce more notions, we specialize the results of \cite{quinn2018asymptotically} in the form that we will need:

\begin{theorem}[{\cite{quinn2018asymptotically}}] \label{thm:quinn_asympt_ksurfaces}
	For every hyperbolic end $E \in \Ends(\Sigma)$ we have
	\[
	\lim_{k \to 0^-} h_k = \lim_{k \to 0^-} (-k) \II_k = h_0,
	\]
	where $h_0$ is the hyperbolic metric in the conformal class at infinity $c_0$. Moreover
	\[
	\bdot{h}_0 = - \frac{1}{2} h_0 - \Re q_0, \qquad \dv{k} \left. (-k)  \II_k \right|_{k = 0}  = 0 ,
	\]
	where $q_0$ is the Schwarzian at infinity of $E$.
\end{theorem}

\begin{corollary} \label{cor:convergence_phik}
	The maps $(\Phi_k)_k$ converge to $\Sch$ $\mathscr{C}^1$-\hsk uniformly over compact subsets, as $k$ goes to $0$.
	\begin{proof}
		First we prove the pointwise convergence. Let $E$ be a hyperbolic end, and consider the path $(\Phi_k(E))_k$ in $T^* \Teich_\Sigma^\conf$. We define $g_k \defin (- k) \II_k$. Then, the relations of Theorem \ref{thm:quinn_asympt_ksurfaces} can be rewritten as follows:
		\[
		g_0 \defin \lim_{k \to 0} g_k = h_0 , \qquad \bdot{h}_0 = - \frac{1}{2} h_0 - \Re q_0 , \qquad \bdot{g}_0 = 0 .
		\]
		The first relation proves that the conformal classes $c_k$ converge to the conformal structure of $\partial_\infty E$. We need to show that the holomorphic quadratic differentials $q_k$ converge to the Schwarzian differential $q_0$. This is a simple application of the relations above, we briefly summarize the steps in the following. First we observe that
		\begin{align*}
		\lim_{k \to 0} 2 \Re q_k & = \lim_{k \to 0} - \frac{2 \sqrt{k + 1}}{k} \left( h_k - \frac{\tr_{g_k}(h_k)}{2} g_k \right) \\
		& = \lim_{k \to 0} - 2 \sqrt{k + 1} \ \frac{h_k - h_0}{k} + \sqrt{k + 1} \ \frac{\tr_{g_k}(h_k) \ g_k - 2 h_0}{k} \\
		& = - 2 \bdot{h}_0 + \dv{k} \left. \tr_{g_k}(h_k) \ g_k \right|_{k = 0} ,
		\end{align*}
		where, in the last step, we are using that $\lim_{k \to 0} \tr_{g_k}(h_k) \ g_k = 2 h_0$. A simple computation shows that $\dv{k} \left. \tr_{g_k}(h_k) \right|_{k = 0} = - 1$. Combining this with the relation above we obtain
		\[
		\lim_{k \to 0} 2 \Re q_k = - 2 \left( - \frac{1}{2} h_0 - \Re q_0 \right) - h_0 + 2 \bdot{g}_0 = 2 \Re q_0 ,
		\]
		which was our claim.
		
		In \cite{quinn2018asymptotically}, the author gave an alternative proof of the existence of the $k$-\hsk surface foliation, for $k$ close to $0$. The strategy of his proof is to apply the Banach implicit function theorem to a function
		\[
		F \vcentcolon (-1,0] \times \Conf^s(\Sigma,c) \longrightarrow \Conf^s(\Sigma,c) ,
		\]
		which satisfies $F(k,\tau) = 0$ if and only if $\tau$ is (a proper multiple of) the metric at infinity associated to the $k$-\hsk surface. Here $\Conf^s(\Sigma,c)$ denotes the space of Sobolev metrics in the conformal class $c$ (see \cite[Theorem~5.1]{quinn2018asymptotically} for details). The map $F$ depends smoothly on $k$ and also on the complex projective structure at infinity $(c,q)$. In particular, the implicit function theorem guarantees the smooth regularity of the metric at infinity $\tau_k$, associated to the $k$-surface $\Sigma_k$, with respect to $k \in (-1,0]$ and $(c,q) \in T^* \Teich_\Sigma^\conf$. Since the tensors $\I_k$ and $\II_k$ are smooth functions of $\tau_k$ and $(c,q)$, the function $\Phi(k;c,q) \defin \Phi_k \circ \Sch^{-1} (c,q)$ is smooth in all its arguments. This properties imply the higher order convergence.
	\end{proof}
\end{corollary}

\subsection{The parametrizations \texorpdfstring{$\Psi_k$}{Psik}} \label{subsec:psik}

The diffeomorphism $\mathcal{H}$ from Theorem \ref{thm:labourie_minimal_lagr1} allows us to convert the family of parametrizations $(\Phi_k)_k$, which take values in $T^* \Teich_\Sigma^\conf$, into a family of parametrizations $(\hat{\Psi}_k)_k$ with values in $\Teich_\Sigma^\hyp \times \Teich_\Sigma^\hyp$. Indeed, the functions
\[
\begin{matrix}
\hat{\Psi}_k \defin \mathcal{H} \circ \Phi_k \vcentcolon & \Ends(\Sigma) & \longrightarrow & \Teich^\hyp_\Sigma \times \Teich^\hyp_\Sigma \\
& [E] & \longmapsto & (h_k, h_k^*) ,
\end{matrix}
\]
associate to each hyperbolic end $E$, the pair of hyperbolic metrics $h_k = (- k) \I_k$ and $h_k^* = - \frac{k}{k + 1} \III_k$ coming from the first and third fundamental forms of the $k$-\hsk surface $\Sigma_k$ of $E$, as we described in Section \ref{subsec:constant_extr_curvature}.

The maps $\hat{\Psi}_k$ have been the main object of study of Bonsante, Mondello and Schlenker in \cite{bonsante2013a_cyclic}, \cite{bonsante2015a_cyclic}. In these works, the authors introduced the notions of \emph{landslide flow} and of \emph{smooth grafting} $\SGr_s'$, and studied their convergence to the classical earthquake flow and grafting map $\Gr$. Our functions $\hat{\Psi}_k$ are actually the inverses of the maps $\SGr_s'$ (the relation between $k$ and $s$ is $k = - \frac{1}{\cosh^2(s/2)}$). 

As the Schwarzian parametrization can be recovered from the limit of the maps $\Phi_k$ when $k \to 0$, the Thurston parametrization can be recovered from the limit of the maps $\hat{\Psi}_k$ when $k \to -1$. Indeed, we have:

\begin{theorem} \label{thm:convergence_psi_k_hat}
	The maps $\hat{\Psi}_k$ converge to $\Th$, as $k$ goes to $-1$, in the following sense: if $E$ is a hyperbolic end, then the length spectrum of $\III_k$ converges to $\iota(\cdot,\mu)$, where $\iota(\cdot,\cdot)$ denotes the geometric intersection of currents. Moreover, the first fundamental forms $\I_k$ converge to the hyperbolic metric of the locally concave pleated boundary $\partial E$.
	\begin{proof}
		Let $E$ be a fixed hyperbolic end. The convergence of the first fundamental forms $\I_k$ is a direct consequence of Theorem \ref{thm:existence_k_foliation}.
		
		By the correspondence between hyperbolic ends and maximal global hyperbolic spatially compact (MGHC) de Sitter spacetimes (see e. g. \cite{mess2007lorentz} and the duality described in Section \ref{subsec:constant_extr_curvature}), the foliation by $k$-\hsk surfaces of $E$ determines a constant curvature surfaces foliation of the MGHC de Sitter spacetime $E^*$ dual of $E$. Through this correspondence, the third fundamental form $\III_k$ of the leaf $\Sigma_k$ in $E$ can be interpreted as the first fundamental form of its dual surface $\Sigma_k^*$ in $E^*$, which has constant intrinsic curvature equal to $\frac{k}{k + 1}$. Moreover, the initial singularity of $E^*$ is dual of the bending measured lamination $\mu$ of the pleated boundary $\partial E$, as shown by \citet[Chapter~3]{benedetti2009canonical_wick}. 
		
		In \cite{belraouti2017asymptotic}, the author studied the intrinsic metrics of families of surfaces which foliate a neighborhood of the initial singularity in $E^*$. In particular, \citet[Theorem~2.10]{belraouti2017asymptotic} proved that, for a wide class of such foliations, the intrinsic metrics of the surfaces converge, with respect to the Gromov equivariant topology, to the real tree dual of the measured lamination $\mu$, as the surfaces approach the initial singularity of $E^*$. By applying this result to the constant curvature foliation of $E^*$, and interpreting $\III_k$ as the first fundamental forms of its leaves, we deduce the convergence of the length spectrum of $\III_k$ to $\iota(\cdot,\mu)$. 
	\end{proof}
\end{theorem}

\subsubsection*{Hyperbolic length functions}

Following \cite{bonsante2015a_cyclic}, we define
\[
\begin{matrix}
j \vcentcolon & \Teich^\hyp_\Sigma \times \Teich_\Sigma^\hyp & \longrightarrow & \R \\
& (h,h^*) & \longmapsto & \int_\Sigma \tr b \dd{a}_h ,
\end{matrix}
\]
which associates, to a normalized pair of hyperbolic metrics $h$, $h'$ with Labourie operator $\mappa{b}{T \Sigma}{T \Sigma}$ (see Definition \ref{def:labourie_operator}), the integral of the trace of $b$ with respect to the area measure of $h$ (here we are identifying, with abuse, the hyperbolic metrics $h$ and $h'$ with their isotopy classes). The quantity $j(h,h')$ satisfies
\[
j(h,h') = 2 \ E(\mappa{\id}{(\Sigma,c)}{(\Sigma,h)}) = 2 \ E(\mappa{\id}{(\Sigma,c)}{(\Sigma,h')}) ,
\]
where $c$ is the conformal class of $h(b \cdot, \cdot)$, and $E(\cdot)$ denotes the energy functional (see \cite[Section~1.2]{bonsante2015a_cyclic}). This shows in particular that $j$ is symmetric, i. e. $j(h,h') = j(h',h)$. 

For any hyperbolic metric $h'$, we define $\mappa{L_{h'}}{\Teich_\Sigma^\hyp}{\R}$ to be $L_{h'}(h) \defin j(h,h')$. The functions $L_{h'}$, which are real analytic by \cite[Proposition~1.2]{bonsante2015a_cyclic}, can be interpreted as generalizations of length functions, in light of the following fact:

\begin{proposition} \label{prop:conergence_length_functions}
	Let $(h_n)_n$, $(h_n^*)_n$ be two sequences of hyperbolic metrics. Suppose that $(h_n)_n$ converges to $h \in \Teich_\Sigma^\hyp$, and that there exists a sequence of positive numbers $(\vartheta_n)_n$ such that the length spectrum of $\varepsilon_n^2 \ h_n^*$ converges to $\iota(\cdot, \mu)$, for some measured lamination $\mu \in \MesLam_\Sigma$. Then
	\[
	\lim_{n \to \infty} \varepsilon_n \ L_{h_n^*} (h_n)  = L_\mu(h) .
	\]
	\begin{proof}
		Using the interpretation via $k$-\hsk surfaces, we can easily prove this statement, which is purely $2$-\hsk dimensional, using $3$-\hsk dimensional hyperbolic geometry.
		
		First we observe that, since the injectivity radius of $h_n^*$ is going to $0$, the sequence $\varepsilon_n$ must converge to $0$. In particular, the limit of $k_n \defin - (\cosh^2 \varepsilon_n)^{-1}$ is equal to $-1$, as $n$ goes to infinity. In \cite[Proposition~6.2]{bonsante2013a_cyclic}, the authors proved that, under our hypotheses, the sequence of hyperbolic ends $(E_n)_n$ given by $E_n \defin \hat{\Psi}_{k_n}^{-1}(h_n,h_n^*)$ (which, in the notation of \cite{bonsante2013a_cyclic}, coincides with $\SGr_{2 \varepsilon_n}'(h_n, h_n^*)$), converges to $E \defin \Gr_\mu(h)$. Recalling the definitions of $h_n$, $h_n^*$, we see that
		\[
		L_{h_n^*}(h_n) = - \frac{k_n}{\sqrt{k_n + 1}} \int_{\Sigma_{k_n}} H_{k_n} \dd{a}_{\I_{k_n}} ,
		\]
		where $\Sigma_{k_n}$ is the $k_n$-\hsk surface inside $E_n$, and $\I_{k_n}$ and $H_{k_n}$ are its first fundamental form and mean curvature, respectively. 
		
		Since $E_n$ goes to $E = \Gr_{\mu}(h)$, and $k_n$ goes to $-1$, the intrinsic metrics of the surfaces $\Sigma_{k_n}$ converge to the hyperbolic metric $h$ of the pleated boundary $\partial E$, and the bending measures of $\partial E_n$ converge to $\mu$. In particular, the integral of the mean curvature of $\Sigma_{k_n}$ converges to $ L_{\mu}(h)$, the length of the bending measure of $\partial E$ (see for instance \cite[Section~2]{mazzoli2018the_dual}). From the relation between $k_n$ and $\varepsilon_n$, we see that
		\[
		\lim_{n \to \infty} \varepsilon_n \left( - \frac{k_n}{\sqrt{k_n + 1}} \right) = 1 .
		\]
		The combination of these two fact implies the statement.
	\end{proof}
\end{proposition}

As done in \cite{bonsante2015a_cyclic}, instead of working directly with $\hat{\Psi}_k$, we we will introduce a family of maps $(\Psi_k)_k$ that have the advantage of taking values in the cotangent space $T^* \Teich_\Sigma^\hyp$. This will be more convenient for the rest of our paper, since we investigate the properties of these parametrizations with respect to the cotangent symplectic structure of $T^* \Teich_\Sigma^\hyp$ and $T^* \Teich_\Sigma^\conf$. The functions $\Psi_k$ are defined as follows:
\[
\begin{matrix}
\Psi_k \vcentcolon & \Ends(\Sigma) & \longrightarrow & T^* \Teich_\Sigma^\hyp \\
& [E] & \longmapsto & (h_k, - \frac{\sqrt{k + 1}}{k} \dd{(L_{h^*_k})}_{h_k}) ,
\end{matrix} 
\]
where $\dd{(L_{h^*_k})}_{h_k}$ denotes the differential of the function $L_{h_k^*}$, defined as before, at the point $h_k$. We also consider the function
\[
\begin{matrix}
\dd{L} \vcentcolon & \Teich^\hyp_\Sigma \times \MesLam_\Sigma & \longrightarrow & T^* \Teich_\Sigma^\hyp \\
& (h,\mu) & \longmapsto & (h, \dd(L_\mu)_h) .
\end{matrix}
\]

\begin{proposition} \label{prop:limit_psik}
	The functions
	\[
	\dd{L} \circ \Th \vcentcolon \Ends(\Sigma) \longrightarrow T^* \Teich_\Sigma^\hyp \quad \text{and} \quad \Psi_k \vcentcolon \Ends(\Sigma) \rightarrow T^* \Teich_\Sigma^\hyp
	\]
	are $\mathscr{C}^1$ diffeomorphisms, for every $k \in (-1,0)$. Moreover, the functions $\Psi_k$ converge pointwisely to $\dd{L} \circ \Th$ as $k$ goes to $-1$.
	\begin{proof}
		A proof of the $\mathscr{C}^1$-\hsk regularity of $\dd{L} \circ \Th$ can be found in \cite[Lemma~1.1]{krasnov2009symplectic}. The smoothness of the maps $\hat{\Psi}_k$ follows from the original work of \citet{labourie1992surfaces}. Up to scalar multiplication in the fiber, the functions $\Psi_k$ are equal to the composition of the $\hat{\Psi}_k$'s with the map
		\[
		\begin{matrix}
		\Teich_\Sigma^\hyp \times \Teich_\Sigma^\hyp & \longrightarrow & T^* \Teich_\Sigma^\hyp \\
		(h,h') & \longrightarrow & (h, \dd{(L_{h'})}_h) .
		\end{matrix}
		\]
		This function has been proved to be a diffeomorphism in \cite[Proposition~1.10]{bonsante2015a_cyclic}. This shows that $\Psi_k$ is a diffeomorphism for every $k \in (-1,0)$.
		The pointwise convergence of the functions $\Psi_k$ follows from Theorem \ref{thm:convergence_psi_k_hat}, Proposition \ref{prop:conergence_length_functions} and the analyticity of the functions $L_{h'}$, established in \cite[Proposition~1.2]{bonsante2015a_cyclic}.
	\end{proof}
\end{proposition}


\section{Volumes and Schl{\"a}fli formulas}

In this section we define two families of volume functions for convex co-\hsk compact hyperbolic $3$-\hsk manifolds: the $W_k$-\hsk volumes, related to the notion of \emph{$W$-\hsk volume} introduced in \cite{krasnov2008renormalized}, and the $\Vol_k^*$-\hsk volumes, related the notion of \emph{dual volume} introduced in \cite{krasnov2009symplectic}. For both these families we will prove a Schl\"afli-\hsk type variation formula, involving the extremal length, in the case of $W_k$, and the hyperbolic length functions $L_{h^*}$ introduced in the previous section, in the case of $\Vol_k^*$. We also describe a simple way to compute the \emph{renormalized volume} $\Vol_R$ of a convex co-\hsk compact hyperbolic manifold using the volumes $W_k$.

\subsection{\texorpdfstring{$W_k$}{Wk}-\hsk volumes}

Let $M$ be a Kleinian manifold and let $\mathscr{G}(M)$ denote the space of convex co-\hsk compact hyperbolic structures of $M$. We define
\[
W_k(M) \defin W(M_k) = \Vol(M_k) - \frac{1}{4} \int_{\partial M_k} H_k \dd{a}_{\I_k} ,
\]
where $M_k$ denotes the compact region of $M$ bounded by the union of the $k$-surfaces sitting inside the ends of $M$. The quantities $\I_k$, $\II_k$, $\III_k$, $c_k$ and $q_k$ of $\partial M_k$ are defined using the conventions of the previous section.

\begin{lemma} \label{lem:differential_Wk_volume}
The function $\mappa{W_k}{\mathscr{G}(M)}{\R}$ satisfies
\[
\dd(W_k)_M(\delta M) = - \Re \scal{q_k}{\delta c_k} .
\]
\begin{proof}
	We apply the variation formula of the $W$-\hsk volume, proved in Proposition \ref{prop:variation_W_volume}. Since the boundary of $M_k$ is a $k$-\hsk surface for every convex co-\hsk compact structure $M$, the term involving $\delta K_e$ vanishes.
	Therefore we have:
	\begin{align*}
		\dd(W_k)_M(\delta M) & = \frac{1}{4} \int_{\Sigma_k} \scall*{\delta \II_k}{\III_k - \frac{H_k}{2} \II_k}_{\II_k} \dd{a_{\I_k}} \\
		& = \frac{1}{4 \sqrt{k + 1}} \int_{\Sigma_k} \scall*{\delta \II_k}{\III_k - \frac{H_k}{2} \II_k}_{\II_k} \dd{a_{\II_k}} \tag{eq. \eqref{eq:area_forms}} \\
		& = - \frac{1}{8} \int_{\Sigma_k} \scall{\delta \II_k}{2 \Re q_k}_{\II_k} \dd{a_{\II_k}} \tag{eq. \eqref{eq:hopf_diff_I_and_III}} \\
		& = - \Re \scal{q_k}{\delta c_k} \tag{Lemma \ref{lem:expression_pairing}} .
	\end{align*}
\end{proof}
\end{lemma}

Starting from Lemma \ref{lem:differential_Wk_volume}, the proof of the Schl\"afli formula for the volumes $W_k$ proceeds in analogy to what done by \citet{schlenker2017notes} for the Schl\"afli formula for the renormalized volume, thanks to the following result:

\begin{theorem}[{Gardiner's formula, \cite[Theorem~8]{gardiner1984measured}}]
	Let $(\Sigma,c)$ be a Riemann surface, and let $\mathcal{F}$ denote the horizontal foliation of a homorphic quadratic differential $q$ of $(\Sigma,c)$. Then the extremal length function $\mappa{\ext_{\mathcal{F}}}{\Teich^\conf_\Sigma}{\R}$ satisfies
	\[
	\dd(\ext_{\mathcal{F}})_c(\delta c) = 2 \Re \scal{q}{\delta c} .
	\]
\end{theorem}

The combination of Lemma \ref{lem:differential_Wk_volume} and the Gardiner's formula immediately implies:

\begin{theorem}[Schl\"afli formula for $W_k$] \label{thm:schlafli_Wk}
	The differential of the function $\mappa{W_k}{\mathscr{G}(M)}{\R}$ can be expressed as follows:
	\[
	\dd{(W_k)}_M(\delta M) = - \frac{1}{2} \dd(\ext_{\mathcal{F}_k})_{c_k}(\delta c_k) ,
	\]
	where $\mathcal{F}_k$ denotes the horizontal foliation of the holomorphic quadratic differential $q_k$.
\end{theorem}

\subsection{The renormalized volume} \label{subsec:renormalized_volume}

The definition of renormalized volume $\Vol_R(M)$ of a conformally compact Einstein manifold $M$ is motivated by the $\AdS$/$\CFT$ correspondence of string theory \cite{witten1998anti_de_sitter}, \cite{graham2000volume_and_area}. \citet{krasnov2008renormalized} enlightened its geometrical meaning in the context of convex co-\hsk compact hyperbolic $3$-\hsk manifolds, describing a regularization procedure based on equidistant foliations from convex subsets of $M$. In relation with the study of the geometry of the Teichm\"uller space, the renormalized volume furnishes a K\"ahler potential for the Weil-\hsk Petersson metric of the Teichm\"uller space, and it allows to give a remarkably simple proof of McMullen's Kleinian reciprocity (see \cite{krasnov2008renormalized} and Section \ref{sec:kleinian_reciprocities}). Moreover, its variation formula has been used by \citet{schlenker2013renormalized} to give a quantitative version of Brock's upper bound of the volume of the convex core of a quasi-\hsk Fuchsian manifold in terms of the Weil-\hsk Petersson distance between the hyperbolic metrics on the boundary of the convex core. 

The aim of this Section is to describe a new and simpler way to define the renormalized volume of a quasi-\hsk Fuchsian manifold in terms of the asymptotic of its foliation by $k$-\hsk surfaces. 

First we recall the Schl\"afli-type formula of the renormalized volume:

\begin{theorem}[{\cite[Lemmas~8.3,8.5]{krasnov2008renormalized}, \cite[Theorem~1.2]{schlenker2017notes}}] \label{thm:schlafli_renormalized}
	The differential of the renormalized volume $\mappa{\Vol_R}{\mathscr{G}(M)}{\R}$ can be expressed as follows:
	\[
	\dd{(\Vol_R)}_M(\delta M) = - \Re \scal{q_0}{\delta c_0} = - \frac{1}{2} \dd{(\ext_{\mathcal{F}_0})}_{c_0}(\delta c_0).
	\]
\end{theorem}

The combination of Corollary \ref{cor:convergence_phik} and Theorem \ref{thm:schlafli_Wk} allows us to give the following description of the renormalized volume $\Vol_R(M)$:

\setcounter{theoremx}{2}

\begin{theoremx}
	The renormalized volume of a quasi-\hsk Fuchsian manifold $M$ satisfies
	\[
	\Vol_R(M) = \lim_{k \to 0^-} \left( W_k(M) - \pi \abs{\chi(\partial M)} \arctanh \sqrt{k + 1} \right) .
	\]
	\begin{proof}
		Let $\widetilde{W}_k(M) \defin W_k(M) - \pi \abs{\chi(\partial M)} \arctanh \sqrt{k + 1}$. We will prove the assertion by showing the following facts:
		\begin{enumerate}[i)]
			\item the differentials of the functions $\widetilde{W}_k$ converge, uniformly over compact subsets of $\QF(\Sigma)$, to the differential of the renormalized volume $\Vol_R$;
			\item the limit, as $k$ goes to $0$, of $\widetilde{W}_k(M)$ coincides with $\Vol_R(M)$ whenever $M$ is Fuchsian.
		\end{enumerate}
		Then the assertion will follow from the connectedness of the space $\QF(\Sigma)$.
		
		The first step easily follows from our previous observations. By Corollary \ref{cor:convergence_phik} and Theorem \ref{thm:schlafli_Wk}, $\dd{\widetilde{W}_k}$ converges, uniformly over compact subsets of $\QF(\Sigma)$, to $- \frac{1}{2} \dd(\ext_{\mathcal{F}_0})(\delta c_0)$, where $\mathcal{F}_0$ is the horizontal foliation of the Schwarzian differential at infinity $q_0$, and $\delta c_0$ is the variation of the conformal structure of $\partial_\infty M$. By Theorem \ref{thm:schlafli_renormalized}, this coincides with $\dd{\Vol_R}$.
		
		It remains to prove the second part of the statement. Let $M$ be a Fuchsian manifold. The equidistant surfaces from the convex core of $M$ at distance $\varepsilon(k) \defin \arctanh \sqrt{k + 1}$ are the two $k$-\hsk surfaces of $M$. Their fundamental forms can be expressed as follows:
		\[
		\I_k = - \frac{1}{k} h, \qquad \II_k = - \frac{k}{\sqrt{k + 1}} h , \qquad \III_k = - \frac{k}{k + 1} h ,
		\]
		where $h$ is the hyperbolic metric on the totally geodesic surface sitting inside $M$. From here, we easily see that:
		\[
		\int_{\Sigma_k} H_k \dd{a_{\I_k}} = 2 \pi \abs{\chi(\partial M)} \sinh 2 \varepsilon(k), \qquad V(M_k) = \pi \abs{\chi(\partial M)} \left( \frac{\sinh 2 \varepsilon(k)}{2} + \varepsilon(k) \right) .
		\]
		In particular, for every Fuchsian manifold $M$, we have
		\[
		W_k(M) = V(M_k) - \frac{1}{4} \int_{\Sigma_k} H_k \dd{a_{\I_k}} = \pi \abs{\chi(\partial M)} \arctanh \sqrt{k + 1} .
		\]
		Therefore the functions $\widetilde{W}_k$ vanish identically over the Fuchsian locus, and the same happens for $\Vol_R(M)$. This concludes the proof of the second step, and therefore of the statement.
	\end{proof}
\end{theoremx}

\begin{remark}
	The quantity $\arctanh \sqrt{k + 1}$ is equal to the distance of the $k$-\hsk surface from the convex core in the Fuchsian case. For a generic quasi-\hsk Fuchsian manifold $M$, the geometric maximum principle \cite[Lemme~2.5.1]{labourie2000une_lemme} shows that the $k$-\hsk surface is at distance less or equal than $\arctanh \sqrt{k + 1}$ from the convex core $\CC M$.
\end{remark}	
	
In the proof that we gave above, we \emph{assumed} the existence of the renormalized volume function $\Vol_R$ and we proved the convergence of the functions $W_k$ to $\Vol_R$. In fact, with some additional work, it is possible to show that the sequence of functions $(\widetilde{W}_k)_k$ is convergent \emph{without} assuming the existence of the function $\Vol_R$. In other words, we can \emph{define} the renormalized volume $\Vol_R(M)$ of a convex co-\hsk compact hyperbolic manifold $M$ as the limit of the sequence $(\widetilde{W}_k(M))_k$.

\subsection{\texorpdfstring{$\Vol_k^*$}{Vk*}-\hsk volumes} 

In analogy to what done for the $W_k$-\hsk volumes, we define 
\[
\Vol^*_k(M) \defin V^*(M_k) = \Vol(M_k) - \frac{1}{2}  \int_{\partial M_k} H_k \dd{a}_{\I_k} .
\]
The Schl\"afli formula for $\Vol_k^*$ is a direct consequence of the variation formula for the dual volume (Proposition \ref{prop:variation_dual_volume}) and the following expression for the variation of the length function $L_{h^*}$:

\begin{lemma}[{\cite[Lemma~7.9]{bonsante2015a_cyclic}}] \label{lem:variation_hyperbolic_length}
	\[
	\dd(L_{h^*})(\delta h) = - \frac{1}{2} \int_\Sigma \scall{\delta h}{h(b \cdot, \cdot) - \tr(b) h}_h \dd{a_h} .
	\]
\end{lemma}

In order to simplify the next statement, we extend the definition of the function $j$ to constant curvature metrics, not necessarily hyperbolic. In particular, if $g$ and $g'$ are Riemannian metrics of with constant Gaussian curvatures $K$ and $K'$, then we set $j(g,g')$ to be $(K K')^{- 1/2} j((- K) g, (- K') g')$ (observe that $(- K) g$ and $(- K') g'$ are hyperbolic). In this way, the function $j$ is $1/2$-\hsk homogeneous in both its arguments, As before, $L_g$ will denote the function $j(g,\cdot)$.

\begin{theorem}[Schl\"afli formula for $\Vol_k^*$] \label{thm:schlafli_Vk*}
	The differential of the function $\mappa{\Vol_k^*}{\mathscr{G}(M)}{\R}$ can be expressed as follows:
	\[
	\dd{(\Vol_k^*)}_M(\delta M) = - \frac{1}{2} \dd{(L_{\III_k})}(\delta \I_k).
	\]
	\begin{proof}
		By Proposition \ref{prop:variation_dual_volume}, the variation of $\Vol_k^*$ verifies
		\[
		\dd{(\Vol_k^*)}_M(\delta M) = \frac{1}{4} \int_{\Sigma_k} \scall{\delta \I_k}{\II_k - H_k \I_k}_{\I_k} \dd{a}_{\I_k} .
		\]
		Using the definitions of $h_k$, $h_k^*$, we can rephrase the expression above as follows:
		\[
		\dd{(\Vol_k^*)}_M(\delta M) = - \frac{\sqrt{k + 1}}{4 k} \int_{\Sigma_k} \scall{\delta h_k}{h_k(b_k \cdot, \cdot) - \tr(b_k) h_k} \dd{a}_{h_k} ,
		\]
		where $b_k = \sqrt{k + 1} \ B_k$ is the Labourie operator between $h_k$ and $h_k^*$ (see Definition \ref{def:labourie_operator}). By Lemma \ref{lem:variation_hyperbolic_length}, the expression above is equal to $\frac{\sqrt{k + 1}}{2 k} \dd{(L_{h_k^*})}_{h_k}(\delta h_k) = - \frac{1}{2} \dd{(L_{\III_k})}_{\I_k}(\delta \I_k)$.
	\end{proof}
\end{theorem}

Let $M$ be a convex co-\hsk compact hyperbolic $3$-\hsk manifold. By Theorem \ref{thm:existence_k_foliation}, the convex subsets $M_k$ approximate, as $k$ goes to $-1$, the \emph{convex core} of $M$, which is the smallest non-\hsk empty convex subset of $M$ (here $C \subseteq M$ is convex if, for every $x, y \in C$, and \emph{for every} geodesic arc $\gamma$ starting at $x$ and ending at $y$, $\gamma$ is entirely contained in $C$). While the surfaces $\partial M_k$ are smoothly embedded in $M$, the boundary of the convex core has the structure of a convex pleated surface, which is only topologically embedded in $M$. However, it is possible to reasonably extend the notion of dual volume to the convex core too, by setting:
\[
V^*_\CC(M) \defin V^*_{-1}(M) = V(\CC M) - \frac{1}{2} L_\mu(h) .
\]
It is not difficult to see that this definition is continuous in $k$, i. e. the limit as $k$ goes to $0$ of the volumes $\Vol_k^*(M)$ is equal to $\Vol_\CC (M)$ (see e. g. \cite[Section~2]{mazzoli2018the_dual}). It turns out that the variation formula of the dual volume of the convex core is, at least formally, the limit of the Sch\"afli formulas of $\Vol_k^*$ as $k$ goes to $-1$, in light of Proposition \ref{prop:limit_psik}. Indeed, we have the following:

\begin{theorem}[{Dual Bonahon-\hsk Schl\"afli formula, \cite{krasnov2009symplectic}, \cite{mazzoli2018the_dual}}] \label{thm:dual_bonahon_schlafli}
	\[
	\dd{(\Vol_\CC^*)}_M(\delta M) = - \frac{1}{2} \dd{(L_\mu)}(\delta h)
	\]
\end{theorem}
This result has been first established by \citet{krasnov2009symplectic} applying the variation formula of the "standard" hyperbolic volume of the convex core, proved by \citet{bonahon1998schlafli}. In a recent work \cite{mazzoli2018the_dual}, we gave a new proof of this result that does not involve the study of the first order variation of the bending measured lamination, which was a highly technical difficulty to handle in the original work \cite{bonahon1998schlafli}.


\section{Volumes and symplectomorphisms}

The aim of this section is to study the properties of the maps $\Phi_k$ and $\Psi_{k'}$. In particular, we will prove that the diffeomorphisms $\mappa{\Phi_k \circ \Psi_{k'}^{-1}}{T^* \Teich_\Sigma^\hyp}{T^* \Teich_\Sigma^\conf}$ are symplectic with respect to the cotangent symplectic structures of $T^* \Teich_\Sigma^\hyp$ and $T^* \Teich_\Sigma^\hyp$, up to a multiplicative factor. This fact extends the results of \citet{krasnov2009symplectic} and \citet{bonsante2015a_cyclic} concerning the grafting map $\Gr$ and the smooth grafting map $\SGr$, respectively. 

\subsubsection*{Relative volumes}

Let $E$ be a hyperbolic end. We denote by $E_k$ the portion of $E$ that is in between the concave pleated boundary $\partial E$ and the $k$-\hsk surface $\Sigma_k$ of $E$. Now we define
\[
w_k(E) \defin \Vol(E_k) - \frac{1}{4} \int_{\Sigma_k} H_k \dd{a}_k + \frac{1}{2} L_\mu(m) ,
\]
where $H_k$ and $\dd{a}_k$ are the mean curvature and the area form of $\Sigma_k$, and $L_\mu(m)$ is the length of the bending measure $\mu$ with respect to the hyperbolic metric $m$ of $\partial E$. Similarly, we define
\[
v_k^*(E) \defin \Vol(E_k) - \frac{1}{2} \int_{\Sigma_k} H_k \dd{a}_k + \frac{1}{2} L_\mu(m) .
\]
The functions $w_k$ and $v_k^*$ are relative versions of the $W_k$-\hsk volume and $\Vol_k^*$-\hsk volume, respectively.

\subsubsection*{Cotangent symplectic structures}

Let $M$ be a smooth $n$-\hsk manifold, with cotangent bundle $\mappa{\pi}{T^* M}{M}$. The \emph{Liouville form} $\lambda$ of $T^* M$ is the $1$-\hsk form defined by:
\[
\lambda_{(p,\alpha)}(v) \defin \alpha (\dd{\pi}_{(p,\alpha)}(v))
\] 
for every $(p,\alpha) \in T^* M$ and $v \in T_{(p,\alpha)} T^* M$. The 2-\hsk form $\omega \defin \dd \lambda$ is non-\hsk degenerate and it defines a natural symplectic structure on the total space $T^* M$. 

\vspace{0.5cm}

In the following, $\lambda^\hyp$, $\lambda^\conf$ will denote the Liouville forms of $T^* \Teich_\Sigma^\hyp$, $T^* \Teich_\Sigma^\conf$, respectively, and $\omega^\hyp$, $\omega^\conf$ their associated symplectic forms. As before, $\Th$ stands for the Thurston parametrization, which we briefly recalled in the end of Section~1. The reader can find the necessary notation concerning the geometry of $k$-\hsk surfaces at the beginning of Section~2, and the definitions of the parametrizations $\Phi_k$ and $\Psi_k$ in Sections \ref{subsec:phik} and \ref{subsec:psik}, respectively.  

The first step of our analysis will be to describe the pullback of the Liouville forms $\lambda^\conf$ and $\lambda^\hyp$ by the maps $\Phi_k$ and $\dd L \circ \Th$, $\Psi_k$, respectively. In particular, we have:

\begin{lemma} \label{lem:pullback_liouville_forms}
	The following relations hold:
	\begin{gather}
	(\Phi_k^* \lambda^\conf)_E(\delta E) = \frac{1}{4} \int_{\Sigma_k} \scall{\delta \II_k}{\Re q_k}_{\II_k} \dd{a}_{\II_k} , \label{eq:pullback_liouville_k_param} \\
	((\dd{L} \circ \Th)^* \lambda^\hyp)_E(\delta E) = \dd{(L_\mu)}_m(\delta m), \label{eq:pullback_liouville_thurston} \\
	(\Psi_k^* \lambda^\hyp)_E(\delta E) = - \frac{1}{2} \int_{\Sigma_k} \scall{\delta \I_k}{\II_k - H_k \I_k }_{\I_k} \dd{a}_{\I_k} , \label{eq:pullback_liouville_hyp_k_param}
	\end{gather}
	where $\delta \I_k$ and $\delta \II_k$ represent the variations of the first and second fundamental forms of the $k$-\hsk surface, respectively.
	\begin{proof}
		The Liouville form $\lambda^\conf$ of $T^* \Teich^\conf_\Sigma$ satisfies
		\[
		(\Phi_k^* \lambda^\conf)_E(\delta E) = \lambda^\conf_{(c_k,q_k)}(\dd{(\Phi_k)}_E(\delta E)) = \Re \scal{q_k}{\delta c_k} ,
		\]
		where $\delta c_k$ is the Beltrami differential representing the variation of $c_k$ as we deform the hyperbolic end along the direction $\delta E$. Then relation \eqref{eq:pullback_liouville_k_param} follows from Lemma \ref{lem:expression_pairing}. 
		
		Relation \eqref{eq:pullback_liouville_thurston} has been originally shown by Krasnov and Schlenker in the proof of \cite[Theorem~1.2]{krasnov2009symplectic}. First observe that the $1$-\hsk form $(\dd{L} \circ \Th)^* \lambda^\hyp$ is well defined since the function $\dd L \circ \Th$ is $\mathscr{C}^1$ (see Proposition \ref{prop:limit_psik}). Similarly to what done above, we see that
		\[
		((\dd{L} \circ \Th)^* \lambda^\hyp)_E(\delta E) = \lambda^\hyp_{(m,\dd{(L_\mu)}_m)}(\dd{(\dd{L} \circ \Th)}_E(\delta E)) = \dd{(L_\mu)}_m (\delta m) ,
		\]
		where $\delta m$ denotes the first order variation of the hyperbolic metric of the concave pleated surface $\partial E$ along the direction $\delta E$.
		
		Finally, the Liouville form $\lambda^\hyp$ satisfies
		\[
		(\Psi_k^* \lambda^\hyp)_E(\delta E) = - \frac{\sqrt{k + 1}}{k} \dd{(L_{h_k^*})}_{h_k}(\delta h_k) .
		\]
		Therefore, relation \eqref{eq:pullback_liouville_hyp_k_param} follows from Lemma \ref{lem:variation_hyperbolic_length} by backtracking the multiplicative factors involved in the definitions of all the quantities.
	\end{proof}
\end{lemma}

Similarly to what done in the previous Section, we can describe the first order variation of the relative volume functions $w_k$ and $v_k^*$ as follows:

\begin{lemma} \label{lem:variation_relative_volumes}
	The relative volumes $w_k$ and $v_k^*$ satisfy:
	\begin{gather*}
	\dd{(w_k)}_E(\delta E) = \frac{1}{4} \int_{\Sigma_k} \scall*{\delta \II_k}{\III_k - \frac{H_k}{2} \II_k}_{\II_k} \dd{a}_{\I_k} + \frac{1}{2} \dd{(L_\mu)}_m(\delta m), \\
	\dd{(v_k^*)}_E(\delta E) = \frac{1}{4} \int_{\Sigma_k} \scall*{\delta \I_k}{\II_k - H_k \I_k}_{\I_k} \dd{a}_{\I_k} + \frac{1}{2} \dd{(L_\mu)}_m(\delta m).
	\end{gather*}
	\begin{proof}
		Both the relations can be proved by applying the same strategy of \cite[Proposition~4.3]{krasnov2009symplectic}. Let $(g_t)_t$ be a differentiable $1$-\hsk parameter family of hyperbolic metrics on $\Sigma \times (0,\infty)$ so that the first order variation of $E_t = (\Sigma \times (0,\infty), g_t)$ coincides with $\delta E$. For any $t$, we choose an embedded surface $S$ in $\Sigma \times (0,\infty)$ that lies below the $k$-\hsk surface of $E_t$ (i. e. it is contained in the interior of the region $(E_t)_k$) for all small values of $t$. Now we decompose the quantity $w_k(E)$ in two terms:
		\[
		w_k(E) = \left( \Vol(N(S,\Sigma_{t,k})) - \frac{1}{4} \int_{\Sigma_{t,k}} H_{t,k} \dd{a}_{k,t} \right) + \left( \Vol(N(\partial E_t,S)) + \frac{1}{2} L_{\mu_t}(m_t) \right)
		\]
		where $\Sigma_{t,k}$ is the $k$-\hsk surface of $E_t$, and $N(S',S'')$ denotes the region of $E$ bounded by $S'$ from below and $S''$ from above.
		
		Following step by step the proof of Proposition \ref{prop:variation_W_volume}, we see that the variation of the first term equals
		\[
		\frac{1}{4} \int_{\Sigma_k} \scall*{\delta \II_k}{\III_k - \frac{H_k}{2} \II_k}_{\II_k} \dd{a}_{\I_k} + \frac{1}{2} \int_S \left( \delta H + \frac{1}{2} \scall{\delta \I}{\II} \right) \dd{a} , 
		\]
		where the mean curvature $H$ and the second fundamental form $\II$ of $S$ are defined with respect to the normal vector field of $S$ pointing \emph{towards} the concave boundary $\partial E$.
		
		The variation formula of the right term can be computed with the exact same argument of \cite{mazzoli2018the_dual}, the only difference is that we are looking at a region bounded by a smooth surface and a locally \emph{concave} pleated surface, while in \cite{mazzoli2018the_dual} we were considering the convex core, which is a region bounded by convex pleated surfaces. This leads to the following variation:
		\[
		\frac{1}{2} \dd{(L_\mu)}_m(\delta m) + \frac{1}{2} \int_S \left( \delta (- H) + \frac{1}{2} \scall{\delta \I}{- \II} \right) \dd{a} .
		\]
		The signs multiplying $H$ and $\II$ are due to the fact that we need to consider the mean curvature and the second fundamental form defined with the normal vector field pointing \emph{outside} of $N(\partial E, S)$, which is the opposite of the one considered above. In particular, when we look at the sum of the two terms, the integrals over $S$ simplify, and we are left with the first relation of our statement.
		
		The second relation follows by an analogous argument, replacing the use of Proposition \ref{prop:variation_W_volume} with Proposition \ref{prop:variation_dual_volume}.
	\end{proof}
\end{lemma}

\begin{lemma} \label{lem:liouville_forms}
	For every $k \in (-1.0)$, we have
	\begin{gather*}
	\dd{w_k} = - \Phi_k^* \lambda^\conf + \frac{1}{2} (\dd{L} \circ \Th)^* \lambda^\hyp , \\
	\dd{v_k^*} = - \frac{1}{2} \Psi_k^* \lambda^\hyp + \frac{1}{2} (\dd{L} \circ \Th)^* \lambda^\hyp .
	\end{gather*}
	\begin{proof}
		The statement is a direct consequence of Lemma \ref{lem:pullback_liouville_forms} and Lemma \ref{lem:variation_relative_volumes}.
	\end{proof}
\end{lemma}

Taking the differential of the identities in Lemma \ref{lem:liouville_forms}, and remembering that $\dd^2 = 0$, we immediately conclude the following:

\begin{theorem} \label{thm:symplectic_maps}
	For every $k \in (-1.0)$, the maps 
	\begin{align*}
	\Phi_k \circ (\dd{L} \circ \Th)^{- 1} & \vcentcolon (T^* \Teich^\hyp_\Sigma,\omega^\hyp) \longrightarrow (T^* \Teich^\conf_\Sigma, 2 \omega^\conf) , \\
	\Psi_k \circ (\dd{L} \circ \Th)^{- 1} & \vcentcolon (T^* \Teich^\hyp_\Sigma,\omega^\hyp) \longrightarrow (T^* \Teich^\hyp_\Sigma, \omega^\hyp) 
	\end{align*}
	are symplectomorphisms.
\end{theorem}

Observe that Theorem \ref{thm:symplectic_maps} is a direct consequence of what we just observed.

\begin{remark}
	Theorem \ref{thm:symplectic_maps}, combined with Corollary \ref{cor:convergence_phik}, implies that the map
	\[
	\Sch \circ (\dd{L} \circ \Th)^{-1} \vcentcolon (T^* \Teich^\hyp_\Sigma,\omega^\hyp) \longrightarrow (T^* \Teich^\conf_\Sigma, 2 \omega^\conf) 
	\]
	is a symplectomorphism, which has been originally shown in \cite[Theorem~1.2]{krasnov2009symplectic}.
	In addition, \cite[Theorem~1.2]{krasnov2009symplectic} and Theorem \ref{thm:symplectic_k} imply also \cite[Theorem~1.11]{bonsante2015a_cyclic}, which states that the function
	\[
	\Sch \circ \Psi_k^{-1} \vcentcolon (T^* \Teich^\hyp_\Sigma,\omega^\hyp) \longrightarrow (T^* \Teich^\conf_\Sigma, 2 \omega^\conf) 
	\]
	is a symplectomorphism for every $k \in (-1,0)$. Finally, by applying Theorem \ref{thm:symplectic_k} to the case $k = k'$, and taking care of the multiplicative factors involved in the definitions of $\Phi_k$ and $\Psi_k$, we deduce that the function 
	\[
	\begin{matrix}
	\hat{\mathcal{H}} \vcentcolon & (T^* \Teich_\Sigma^\conf, \omega^\conf) & \longrightarrow & (T^* \Teich_\Sigma^\hyp, \omega^\hyp) \\
	& (c,q) & \longrightarrow & (h(c,q), \dd{(L_{h(c,-q)})})
	\end{matrix}
	\]
	is a symplectomorphism, where $h(c, \pm q) = \varphi_c^{-1}(\pm q)$ is the hyperbolic metric of $\Sigma$ for which the identity map $(\Sigma,c) \rightarrow (\Sigma, h(c,\pm q))$ is harmonic with Hopf differential equal to $\pm q$  (see Theorem \ref{thm:wolf_param}). 
\end{remark}


\section{Kleinian reciprocities} \label{sec:kleinian_reciprocities}

Let $M$ be a Kleinian manifold and let $\mathscr{G}(M)$ denote the space of convex co-\hsk compact hyperbolic structures of $M$. Any isotopy class of hyperbolic metrics $[g] \in \mathscr{G}(M)$ has a collection of $k$-\hsk surfaces, each one sitting inside a hyperbolic end $E_i$ of $(M,g)$. In this way, we can define a function 
\[
\phi_k \vcentcolon \mathscr{G}(M) \longrightarrow T^* \Teich^\conf_{\partial M} ,
\]
which associates to any class $[g]$ the data $(\Phi_k(E_i))_i$ of its $k$-\hsk surfaces. Similarly, we define the function $\mappa{\psi_k}{\mathscr{G}(M)}{T^* \Teich^\hyp_{\partial M}}$, sending $[g]$ into the data $(\Psi_k(E_i))_i$.

\begin{theoremx}
	For every $k \in (-1,0)$, the image $\phi_k(\mathscr{G}(M))$ (resp. $\psi_k(\mathscr{G}(M))$) is a Lagrangian submanifold of $T^* \Teich_\Sigma^\conf$ (resp. $T^* \Teich_\Sigma^\hyp$).
	\begin{proof}
		The statement is a consequence of Lemma \ref{lem:liouville_forms} and of the variation formula of the dual volume of a convex co-\hsk compact hyperbolic manifold. To see this, first we apply Lemma \ref{lem:liouville_forms} to each end of $M$:
		\[
		\dd{w_{k,i}} = - \Phi_{k,i}^* \lambda^\conf_i + \frac{1}{2} (\dd{L_i} \circ \Th_i)^* \lambda^\hyp_i .
		\]
		By the dual Bonahon-\hsk Schl\"afli formula (Theorem \ref{thm:dual_bonahon_schlafli}), we have that
		\[
		\dd{\Vol^*} (\delta E) = - \frac{1}{2} \sum_i \dd(L_{\mu_i})(\delta m_i) = - \frac{1}{2} \sum_i ((\dd{L_i} \circ \Th_i)^* \lambda^\hyp_i) (\delta E) .
		\]
		Therefore we deduce that
		\[
		\dd\left( \sum_i w_{k,i} + \Vol^* \right) = \sum_i \dd{w_{k,i}} + \dd{\Vol^*} = - \phi_k^* \left( \sum_i \lambda^\conf_i \right) = - \phi_k^* \lambda^\conf .
		\]
		The function $\sum_i w_{k,i} + \Vol^*$ turns out to be equal to the $W$-\hsk volume of $M_k$, the portion of $M$ contained in the union of the $k$-\hsk surfaces of the ends $(E_i)_i$. Indeed:
		\begin{align*}
		\sum_i w_{k,i}(E_i) + \Vol^*(M) & = \sum_i \left( \Vol(E_{k,i}) - \frac{1}{4} \int_{\Sigma_{k,i}} H_{k,i} \dd{a}_{k,i} + \frac{1}{2} L_{\mu_i}(m_i) \right) + \Vol(\CC M) - \frac{1}{2} L_\mu(m) \\
		& = \Vol(\CC M) + \sum_i \Vol(E_{k,i}) - \frac{1}{4} \sum_i \int_{\Sigma_{k,i}} H_{k,i} \dd{a}_{k,i} + \frac{1}{2} L_\mu(m) - \frac{1}{2} L_\mu(m) \\
		& = \Vol(M_k) - \frac{1}{4} \int_{\partial M_k} H_k \dd{a}_k \\
		& = W_k(M) .
		\end{align*}
		Therefore we have proved that $\dd{W_k} = - \phi_k^* \lambda^\conf$. Taking the differential of this identity we obtain that $\phi_k^* \omega^\conf = 0$. This implies the statement, since $\phi_k$ is an embedding and $2 \dim \mathscr{G}(M) = \dim T^* \Teich^\conf_{\partial M}$.
		
		In an analogous manner we can prove that $\psi_k^* \lambda^\hyp = - 2 \dd{\Vol_k^*}$. To see this, it is enough to replace the role of the relative $W$-\hsk volumes $w_{k,i}$ with the dual volumes $v_{k,i}^*$ and then proceed in the exact same way. Again, by taking the differential of the identity $\psi_k^* \lambda^\hyp = - 2 \dd{\Vol_k^*}$, we obtain the second part of the statement.
	\end{proof}
\end{theoremx}

Theorem \ref{thm:reciprocities} is a generalization of Krasnov and Schlenker's reformulation of McMullen's Kleinian reciprocity Theorem \cite[Theorem~1.5]{krasnov2009symplectic}, and their result can be recovered by taking the limit of the identity $\phi_k^* \omega^\conf = 0$ and applying Corollary \ref{cor:convergence_phik} to each hyperbolic end of $M$. Moreover, \citet[Theorem~1.4]{krasnov2009symplectic} proved that the image of the function $\dd L \circ \Th$ is Lagrangian inside $(T^* \Teich_\Sigma^\hyp, \omega^\hyp)$. Since the map $\dd L \circ \Th$ is the limit of the $\psi_k$'s, the part of the statement concerning the maps $\psi_k$ can be similarly seen as an extension of Krasnov and Schlenker's original result.

\subsection{Quasi-\hsk Fuchsian reciprocities} \label{subsec:quasi_fuchsian_reciprocity}

In this section we present a generalization of McMullen's quasi-\hsk Fuchsian reciprocity Theorem in its original formulation from \cite{mcmullen1998complex}. First we will recall McMullen's original statement, and then we will see how to formulate Theorem \ref{thm:reciprocities} is a similar manner. We define the Bers' embeddings to be the maps:
\begin{align*}
\begin{matrix}
\beta_X \vcentcolon & \Teich_{\overline{\Sigma}} & \longrightarrow & T^*_X \Teich_\Sigma \\
& Y & \longmapsto & \Sch(Q(X,Y))^+
\end{matrix}
& &
\begin{matrix}
\beta_Y \vcentcolon & \Teich_\Sigma & \longrightarrow & T^*_Y \Teich_{\overline{\Sigma}} \\
& Y & \longmapsto & \Sch(Q(X,Y))^-
\end{matrix}
\end{align*}
where $Q(X,Y)$ denotes the unique quasi-\hsk Fuchsian manifold with conformal classes at infinity $(X,Y)$, and $\Sch(Q(X,Y))^\pm$ are the Schwarzian differentials at infinity on the upper and lower boundaries at infinity. McMullen's original formulation of the quasi-\hsk Fuchsian reciprocity Theorem is the following:

\begin{theorem}[{\cite[Theorem~1.6]{mcmullen1998complex}}]
	Given $(X,Y) \in \Teich_\Sigma \times \Teich_{\overline{\Sigma}}$, the differentials of the Bers' embeddings
	\[
	\dd{(\beta_X)_Y} \vcentcolon T_Y \Teich_{\overline{\Sigma}} \longrightarrow T^*_X \Teich_\Sigma , \qquad \dd{(\beta_Y)_X} \vcentcolon T_X \Teich_\Sigma \longrightarrow T^*_Y \Teich_{\overline{\Sigma}}
	\]
	are adjoint linear operators. In other words, $\dd{(\beta_X)_Y} = \dd{(\beta_Y){}^*_X}$.
\end{theorem}

We want to analogous statements in the case in which $\Sch$ is replaced by $\phi_k$ or $\psi_k$. For every $k \in (-1,0)$, let $B_k^\conf$ and $B_k^\hyp$ be the maps
\begin{align*}
\begin{matrix}
B_k^\conf \vcentcolon & \QF_\Sigma & \longrightarrow & \Teich_\Sigma^\conf \times \Teich_{\overline{\Sigma}}^\conf \\
& M & \longmapsto & (c_k^+, c_k^-)
\end{matrix}
& &
\begin{matrix}
B_k^\hyp \vcentcolon & \QF_\Sigma & \longrightarrow & \Teich_\Sigma^\hyp \times \Teich_{\overline{\Sigma}}^\hyp \\
& M & \longmapsto & (h_k^+, h_k^-)
\end{matrix}
\end{align*}
where $c_k^\pm$ are the conformal classes of the second fundamental forms of the upper and lower $k$-\hsk surface of $M$, respectively, and $h_k^\pm$ are the hyperbolic metrics $(-k) \I_k^\pm$ of the upper and lower $k$-\hsk surface of $M$, respectively.

A consequence of Labourie and Schlenker's works \cite{labourie1991probleme}, \cite{schlenker2006hyperbolic} (see also \cite[Theorem~4.2]{mazzoli2019dual_volume_WP} for details) is that the function $B_k^\hyp$ is a diffeomorphism for every $k \in (-1,0)$. We do not know if the same is true for $B_k^\conf$, we will assume this to be true for the rest of this section. In analogy to Bers' embeddings, we define the following maps:

\begin{align*}
\begin{matrix}
\beta_{k,X}^\conf \vcentcolon & \Teich_{\overline{\Sigma}}^\conf & \longrightarrow & T^*_X \Teich_\Sigma^\conf \\
& Y & \longmapsto & \phi_k^+ \circ (B_k^\conf)^{-1} (X,Y)
\end{matrix}
& &
\begin{matrix}
\beta_{k,Y}^\conf \vcentcolon & \Teich_\Sigma^\conf & \longrightarrow & T^*_Y \Teich_{\overline{\Sigma}}^\conf \\
& Y & \longmapsto & \phi_k^- \circ (B_k^\conf)^{-1}(X,Y)
\end{matrix} \\
\begin{matrix}
\beta_{k,X}^\hyp \vcentcolon & \Teich_{\overline{\Sigma}}^\hyp & \longrightarrow & T^*_X \Teich_\Sigma^\hyp \\
& Y & \longmapsto & \psi_k^+ \circ (B_k^\hyp)^{-1}(X,Y)
\end{matrix}
& &
\begin{matrix}
\beta_{k,Y}^\hyp \vcentcolon & \Teich_\Sigma^\hyp & \longrightarrow & T^*_Y \Teich_{\overline{\Sigma}}^\hyp \\
& Y & \longmapsto & \psi_k^- \circ (B_k^\hyp)^{-1}(X,Y)
\end{matrix}
\end{align*}
where:
\begin{enumerate}[a)]
	\item $(B_k^\conf)^{-1} (X,Y)$ is the conjecturally unique quasi-\hsk Fuchsian manifold whose upper and lower $k$-\hsk surfaces have $X$ and $Y$ as conformal classes of their second fundamental forms, respectively;
	\item $(B_k^\hyp)^{-1} (X,Y)$ is the unique quasi-\hsk Fuchsian manifold whose upper and lower $k$-\hsk surfaces have $X$ and $Y$ as hyperbolic structures induced by their first fundamental forms, respectively;
	\item $\phi_k^\pm \circ (B_k^\conf)^{-1} (X,Y)$ are the holomorphic quadratic differentials $q_k^\pm$ on the upper and lower $k$-\hsk surfaces (as defined in Section \ref{subsec:phik});
	\item $\psi_k^\pm \circ (B_k^\hyp)^{-1}(X,Y)$ are the $1$-\hsk forms $\dd{(L_{\III_k^\pm})_{\I_k^\pm}}$ on the upper and lower $k$-\hsk surfaces (as defined in Section \ref{subsec:psik}).
\end{enumerate}

Now that we have introduced all the notation, we are ready to state the formulations of the quasi-\hsk Fuchsian reciprocity Theorems that follow from Theorem \ref{thm:reciprocities}:

\begin{theorem} \label{thm:quasi_fuch_recipr_conf}
	For every $(X,Y) \in \Teich_\Sigma^\hyp \times \Teich_{\overline{\Sigma}}^\hyp$, the differentials of the maps
	\[
	\dd{(\beta_{k,X}^\hyp)}_Y \vcentcolon T_Y \Teich_{\overline{\Sigma}}^\hyp \longrightarrow T^*_X \Teich_\Sigma^\hyp , \qquad \dd{(\beta_{k,Y}^\hyp)}_X \vcentcolon T_X \Teich_\Sigma^\hyp \longrightarrow T^*_Y \Teich_{\overline{\Sigma}}^\hyp
	\]
	are adjoint linear operators.
\end{theorem}

\begin{theorem} \label{thm:quasi_fuch_recipr_hyp}
	If the map $B_k^\conf$ is a diffeomorphism, then for every $(X,Y) \in \Teich_\Sigma^\conf \times \Teich_{\overline{\Sigma}}^\conf$ , the differentials of the maps
	\[
	\dd{(\beta_{k,X}^\conf)}_Y \vcentcolon T_Y \Teich_{\overline{\Sigma}}^\conf \longrightarrow T^*_X \Teich_\Sigma^\conf , \qquad \dd{(\beta_{k,Y}^\conf)}_X \vcentcolon T_X \Teich_\Sigma^\conf \longrightarrow T^*_Y \Teich_{\overline{\Sigma}}^\conf
	\]
	are adjoint linear operators.
\end{theorem}

\begin{proof}[Proof of Theorems \ref{thm:quasi_fuch_recipr_conf} and \ref{thm:quasi_fuch_recipr_hyp}]
	Let $\mappa{F}{N^+ \times N^-}{T^*(N^+ \times N^-)}$ be a smooth function satisfying $\pi \circ F = \id$, where $\mappa{\pi}{T^* (N^+ \times N^-)}{N^+ \times N^-}$ is the cotangent bundle projection. For every $X$ in $N^+$, we set $\mappa{F^+_X}{N^-}{T_X^* N^+}$ to be $F^+_X(Y) \defin F(X,Y)^+$, where $F(X,Y)^+$ is the component of $F(X,Y)$ in the fiber $T_X^* N^+$, and, for every $Y \in N_-$, we set $\mappa{F^-_Y}{N^+}{T^*_Y N^-}$ to be $F^-_Y(X) \defin F(X,Y)^-$, where $F(X,Y)^-$ is the component of $F(X,Y)$ in the fiber $T_Y^* N^-$. Then, the following relation holds:
	\[
	\scal{\dd{(F^-_Y)}_X(u)}{v} - \scal{\dd{(F^+_X)}_Y(v)}{u} = (F^* \omega)_{(X,Y)} ((u,0), (0,v)) ,
	\]
	for all $(X,Y) \in N^+ \times N^-$, $u \in T_X N^+$, $v \in T_Y N^-$. A proof of this relation can be found in \cite{} for the function $F = \Sch$, the proof of the general case is formally identical. Now, using this relation for the maps $F = \phi_k \circ (B_k^\conf)^{-1}$ and $F = \psi_k \circ (B_k^\hyp)^{-1}$, and applying Theorem \ref{thm:reciprocities}, we obtained the desired statement.
\end{proof}


\section{The \texorpdfstring{$k$}{k}-\hsk flows are Hamiltonian}

In this section we show that the $k$-\hsk surface foliation of a hyperbolic end can be described as the integral curve of a time-\hsk dependent Hamiltonian vector field with respect to the symplectic structure $2 \Phi_k^* \omega^\conf = \Psi_k^* \omega^\hyp$ on $\Ends(\Sigma)$, which does not depend on $k$ in light of Theorem \ref{thm:symplectic_k}. The vector fields we will look at are defined in terms of the diffeomorphisms $(\Phi_k)_k$ and $(\Psi_k)_k$ as follows:
\begin{align*}
X_k & \defin \dv{h} \left. \Phi_{k + h} \circ \Phi_k^{- 1} \right|_{h = 0}  \in \Gamma(T (T^* \Teich_\Sigma^\conf)) , \\
Y_k & \defin \dv{h} \left. \Psi_{k + h} \circ \Psi_k^{- 1} \right|_{h = 0}  \in \Gamma(T (T^* \Teich_\Sigma^\hyp)) .
\end{align*}
In order to simplify the notation, whenever we have an object $X$ that depends on the curvature $k$, we will denote by $\bdot{X}$ its derivative with respect to $k$. We denote by $\mappa{m_k}{\Ends(\Sigma)}{\R}$ the function
\[
m_k(E) \defin \int_{\Sigma_k} H_k \dd{a_{\I_k}} .
\]

\begin{lemma}
	For every $k \in (-1,0)$, we have
	\begin{align}
	\lambda^\conf(X_k) \circ \Phi_k & = - \bdot{w}_k + \frac{1}{8 (k + 1)} \ m_k , \label{eq:liouville_time_dep_vect_conf} \\
	\lambda^\hyp(Y_k) \circ \Psi_k & = - 2 \bdot{v}_k^* + \frac{1}{2 k} \ m_k . \label{eq:liouville_time_dep_vect_hyp}
	\end{align}
	\begin{proof}
		Let $E$ be a fixed hyperbolic end. If $(c_k,q_k)$ denotes the point $\Phi_k(E) \in T^* \Teich_\Sigma^\conf$, then the Liouville form $\lambda^\conf$ satisfies
		\[
		\lambda^\conf(X_k) \circ \Phi_k(E) = (\lambda^\conf)_{\Phi_k(E)} \left( \dv{h} \left. \Phi_{k + h}(E) \right|_{h = 0} \right) = \Re \scal{q_k}{\bdot{c}_k} .
		\]
		By Proposition \ref{prop:variation_W_volume}, we have
		\begin{align*}
		\bdot{w}_k(E) & = \frac{1}{4} \int_{\Sigma_k} \left( \scall*{\bdot{\II}_k}{\III_k - \frac{H_k}{2}\II_k}_{\II_k} + \frac{H_k}{2 (k + 1)} \right) \dd{a}_{\I_k} \\
		& = - \Re \scal{q_k}{\bdot{c}_k} + \frac{1}{8 (k + 1)} \int_{\Sigma_k} H_k \dd{a}_{\I_k} \\
		& = - \Re \scal{q_k}{\bdot{c}_k} + \frac{1}{8(k + 1)} \ m_k(E) .
		\end{align*}
		Combining these two relations we obtain the first part of the statement. Similarly, we see that
		\[
		\lambda^\hyp(Y_k) \circ \Psi_k(E) = - \frac{\sqrt{k + 1}}{k} \dd{(L_{h_k^*})}_{h_k}(\bdot{h}_k) .
		\]
		By definition of $h_k$, we have $\bdot{h}_k = - \I_k - k \ \bdot{\I}_k$. Using Lemma \ref{lem:variation_hyperbolic_length}, we obtain
		\begin{align*}
			\lambda^\hyp(Y_k) \circ \Psi_k(E) & = \frac{1}{2 k} \int_{\Sigma_k} \scall{- \I_k - k \ \bdot{\I}_k}{\II_k - H_k \I_k}_{\I_k} \dd{a_{\I_k}} \\
			& = - \frac{1}{2} \int_{\Sigma_k} \scall{\bdot{\I}_k}{\II_k - H_k \I_k}_{\I_k} \dd{a_{\I_k}} + \frac{1}{2 k} \int_{\Sigma_k} H_k \dd{a_{\I_k}} \\
			& = - 2 \bdot{v}_k^*(E) + \frac{1}{2 k} \ m_k(E) ,
		\end{align*}
		where, in the last step, we used Proposition \ref{prop:variation_dual_volume}.
	\end{proof}
\end{lemma}

\begin{lemma} \label{lem:deriv_pullback_liouville_time_dep}
	Let $M$ and $N$ be a $n$- and a $2n$-\hsk manifold, respectively, and let $\mappa{\varphi_t}{N}{T^* M}$ be a $1$-\hsk parameter family of diffeomorphisms, indexed by a variable $t$ varying in an open interval $J$ of $\R$. Denote by $\lambda$ the Liouville form of $T^* M$, and set $V_t$ to be the vector field of $T^* M$ given by
	\[
	V_t \defin \dv{h} \left. \varphi_{t + h} \circ \varphi_t^{- 1} \right|_{h = 0} ,
	\]
	for any $t \in J$. Then we have
	\[
	(\varphi_t^{-1})^* \left( \dv{t} \varphi_t^* \lambda \right) = \iota_{V_t} \omega + \dd(\iota_{V_t} \lambda) ,
	\]
	for every $t \in J$.
	\begin{proof}
		The statement is a consequence of Cartan formula. The time-\hsk dependent family of vector fields $(V_t)_t$ corresponds to a ordinary vector field $\tilde{V}$ on the manifold $J \times T^* M$, by setting
		\[
		\tilde{V}(t,\cdot) \defin \partial_t + V_t(\cdot) \in T_t J \times T_\cdot (T^* M) \cong T_{(t,\cdot)} (J \times T^* M) .
		\]
		An intergral curve $\gamma = \gamma(t)$ of $(V_t)_t$ in $T^* M$ corresponds to the integral curve $t \mapsto (t,\gamma(t))$ of $\tilde{V}$ in $J \times T^* M$. Let $\pi$ denote the projection of $J \times T^* M$ onto its second component. We apply Cartan formula to the $1$-\hsk form $\pi^* \lambda$ and the vector field $\tilde{V}$, obtaining
		\begin{equation} \label{eq:cartan_formula}
			\mathcal{L}_{\tilde{V}} \pi^* \lambda = \iota_{\tilde{V}} \dd( \pi^* \lambda ) + \dd(\iota_{\tilde{V}} \pi^* \lambda) . 
		\end{equation}
		A straightforward computation proves the following relations:
		\begin{gather*}
			\left. \iota_{\tilde{V}} \dd( \pi^* \lambda ) \right|_{(t,\cdot)} = \left. \pi^* (\iota_{V_t} \dd \lambda) \right|_{(t,\cdot)} , \\
			(\iota_{\tilde{V}} \pi^* \lambda) (t,\cdot) = (\iota_{V_t} \lambda \circ \pi) (t,\cdot) , \\
			\left. (\varphi_t^{-1} \circ \pi)^* \left( \dv{t} \varphi_t^* \lambda \right) \right|_{(t,\cdot)} = \left. \mathcal{L}_{\tilde{V}} \pi^* \lambda \right|_{(t,\cdot)} .
		\end{gather*}
		Replacing these expressions in the equation \eqref{eq:cartan_formula}, we obtain that, for every $t \in J$
		\begin{align*}
		\left. \pi^* \left( (\varphi_t^{-1})^* \left( \dv{t} \varphi_t^* \lambda \right) - \iota_{V_t} \omega - \dd(\iota_{V_t} \lambda) \right) \right|_{(t,\cdot)} = 0 .
		\end{align*}
		Since $\dd{\pi}_{(t,\cdot)}$ is surjective, the pullback by $\pi$ at $(t,\cdot)$ is injective on $k$-\hsk forms. In particular, for every $t \in J$ we must have
		\[
		(\varphi_t^{-1})^* \left( \dv{t} \varphi_t^* \lambda \right) - \iota_{V_t} \omega - \dd(\iota_{V_t} \lambda) = 0 ,
		\]
		which proves the statement.
	\end{proof}
\end{lemma}

\begin{theoremx}
	For every $k \in (-1,0)$, the vector field $X_k$ of $T^* \Teich_\Sigma^\conf$ is Hamiltonian with respect to the symplectic structure $\omega^\conf$, with Hamiltonian function $- \frac{1}{8(k + 1)} \ m_k \circ \Phi_k^{-1}$. Similarly, the vector field $Y_k$ of $T^* \Teich_\Sigma^\hyp$ is Hamiltonian with respect to the symplectic structure $\omega^\hyp$, with Hamiltonian function $- \frac{1}{2 k} \ m_k \circ \Psi_k^{-1}$.
	
	\begin{proof}
		From Lemma \ref{lem:liouville_forms} we see that
		\begin{equation} \label{eq:deriv_pullback_liouville}
		\dv{k} \Phi_k^* \lambda^\conf = \dv{k} \left[ - \dd{w_k} + \frac{1}{2} (\dd{L} \circ \Th)^* \lambda^\hyp \right] = - \dd{\bdot{w}_k} .
		\end{equation}
		Applying Lemma \ref{lem:deriv_pullback_liouville_time_dep} to $N = \Ends(\Sigma)$, $M = \Teich_\Sigma^\conf$ and $\varphi_t = \Phi_k$, we get
		\begin{equation} \label{eq:deriv_pullback_liouville_time_dep}
		(\Phi_k^{-1})^* \left( \dv{k} \Phi_k^* \lambda^\conf \right) = \iota_{X_k} \omega^\conf + \dd(\iota_{X_k} \lambda^\conf) .
		\end{equation}
		Now, putting everything together, we obtain
		\begin{align*}
		\iota_{X_k} \omega^\conf & = (\Phi_k^{-1})^* \left( \dv{k} \Phi_k^* \lambda^\conf \right) - \dd(\iota_{X_k} \lambda^\conf) \tag{eq. \eqref{eq:deriv_pullback_liouville_time_dep}} \\
		& = - (\Phi_k^{-1})^* \dd{\bdot{w}_k} - \dd \left( - \bdot{w}_k \circ \Phi_k^{-1} + \frac{1}{8 (k + 1)} \ m_k \circ \Phi_k^{-1} \right) \tag{eq \eqref{eq:liouville_time_dep_vect_conf} and \eqref{eq:deriv_pullback_liouville}} \\
		& = - \dd(\bdot{w}_k \circ \Phi_k^{-1}) + \dd(\bdot{w}_k \circ \Phi_k^{-1}) - \frac{1}{8 (k + 1)} \dd(m_k \circ \Phi_k^{-1}) \\
		& = - \frac{1}{8 (k + 1)} \dd(m_k \circ \Phi_k^{-1}) ,
		\end{align*}
		which proves the first part of the statement. With the exact same strategy we can prove the assertion concerning the vector fields $(Y_k)_k$.
	\end{proof}
\end{theoremx}

\begin{remark} \label{rmk:multipl_constant}
	It can be easily checked that the choice of the multiplicative constant in the definition of $q_k$, and consequently of $\Phi_k$, becomes relevant for Theorem \ref{thm:kflow_hamiltonian} to hold. The same holds for the multiplicative constant in the definition of $\Psi_k$.
\end{remark}


\appendix

\section{}

\begin{lemma} \label{lem:expression_pairing}
	Let $(g_t)_t$ be a $1$-\hsk parameter family of Riemannian metrics on $\Sigma$, with conformal classes $c_t = [g_t]$. If $\delta c$ denotes the Beltrami differential representing the variation of the conformal classes $(c_t)_t$, and $\delta g$ the variation of the Riemannian metrics $(g_t)_t$, then we have
	\[
	\Re \scal{q}{\delta c} = \frac{1}{4} \int_\Sigma \scall{\delta g}{\Re q}_g \dd{a}_g .
	\]
	\begin{proof}
		Let $(g_t)_t$ be a smooth $1$-\hsk parameter family of metrics so that the conformal class of $g = g_0$ is equal to $c = c_0$, and the derivative at $t = 0$ of the conformal class $c_t$ of $g_t$ coincides with $\delta c$. If $X_t$ is an object that depends on $t$, then $\delta X$ will denote its derivative with respect to $t$ at $t = 0$. Let $J_t$ be the almost complex structure of $g_t$ for every $t$. As shown in \cite[Section~2.1]{bonsante2015a_cyclic}, the Beltrami differential $\nu_t$ of the map $\mappa{\id}{(\Sigma, c)}{(\Sigma, c_t)}$ satisfies
		\[
		\nu_t = (\1 - J_t J)^{-1}(1 + J_t J) ,
		\]
		In particular its derivative $\delta \nu$ can be expressed as $\frac{1}{2} \delta J \ J$. The almost complex structure $J_t$ of $g_t$ is characterized by the relation $\dd{a}_t(\cdot, \cdot) = g_t( J_t \cdot, \cdot)$, where $\dd{a}_t$ is the area form of the metric $g_t$.
		Taking the derivative of this identity, and using the fact that $\dd{a}_g = \sqrt{ \det(g_{i j}) } \dd{x}^1 \wedge \dd{x}^2$ in local coordinates, we obtain
		\[
		\frac{1}{2} \scall{\delta g}{g}_g \dd{a} = \frac{1}{2} \tr(g^{-1} \delta g) \dd{a} = \delta(\dd{a}_t) = \delta(g_t( J_t \cdot, \cdot)) = \delta g (J \cdot, \cdot) + g(\delta J \cdot, \cdot) .
		\]
		If $\delta g = g(A \cdot, \cdot)$, with $A$ $g$-\hsk self-\hsk adjoint, then from the relation above we see that
		\[
		\delta J \ J = A - \frac{1}{2} \tr(A) \1 = A_0 ,
		\]
		where $A_0$ stands for the traceless part of $A$. In particular, this proves that $\delta \nu = \frac{1}{2} A_0$. The pairing between Beltrami differentials and holomorphic quadratic differentials can be described as follows:
		\[
		\scal{q}{\mu} \defin \int_\Sigma q \bullet \mu ,
		\]
		where $q \bullet \mu$ is the $\C$-\hsk valued $2$-\hsk form given by
		\[
		(q \bullet \mu) (u,w) \defin \frac{1}{2 i} (q(\mu (u), w) - q(u, \mu(w))) .
		\]
		Again, we refer to \cite[Section~2.1]{bonsante2015a_cyclic} for a more detailed description. Let now $B$ be the traceless and $g$-\hsk self-\hsk adjoint operator satisfying $\Re q (\cdot, \cdot) = g(B \cdot, \cdot)$. Given any unit vector $u$, the basis $u$, $J u$ is orthonormal and positive oriented. In particular, since $q \bullet \delta \nu$ is a multiple of the volume form $\dd{a}_g$ ($\Sigma$ is a $2$-\hsk manifold), we must have $q \bullet \delta \nu = (q \bullet \delta \nu)(u, J u) \dd{a}_g$. Now we observe:
		\begin{align*}
			\Re (q \bullet \delta \nu)(u, J u) & = \Re \frac{1}{2 i} (q(\delta \nu (u), J u) - q(u,\delta \nu(J u))) \\
			& = \Re \frac{1}{2 i} (i q(\delta \nu (u), u) + q( J^2 u,\delta \nu(J u))) \tag{$q$ $\C$-\hsk linear and $J^2 = - \1$} \\
			& = \Re \frac{1}{2} (q(\delta \nu (u), u) + q(\delta \nu (J u), J u)) \tag{$q$ $\C$-\hsk linear} \\
			& = \frac{1}{4} (g(B A_0 u,u) + g(B A_0 J u, J u)) \tag{def. of $B$ and $\delta \nu = \frac{1}{2} A_0$} \\
			& = \frac{1}{4} \tr(B A_0) \tag{$u$, $J u$ orthon. basis}	\\
			& = \frac{1}{4} \tr(B A) \tag{$B$ traceless} \\
			& = \frac{1}{4} \scall{\Re q}{\delta g} \tag{def.s of $A$ and $B$}
		\end{align*}
		Combining what we have proved so far, we obtain that
		\[
		\Re \scal{q}{\delta c} = \int_\Sigma \Re q \bullet \delta \nu = \frac{1}{4} \int_\Sigma \scall{\Re q}{\delta g} \dd{a}_g ,
		\]
		which is our desired relation.
	\end{proof}
\end{lemma}

In the following, we are going to recall the variation formula for the dual volume of hyperbolic $3$-\hsk manifolds, and we will prove the expression of the variation formula of the $W$-\hsk volume that we apply in this paper. First, in order to simplify the notation, we are going to denote by $\dd{a}$ the area form of the \emph{first} fundamental form $\I$ and by $\scall{A}{B}$ the scalar product between $(1,1)$-\hsk tensors induced by the metric $\I$.

In our paper we consider two kinds of variation. One possibility is to fix a smooth compact $3$-\hsk manifold with boundary $N$ and vary its hyperbolic metrics $(g_t)_t$. Each Riemannian metric $g_t$ determines an induced metric $\I_t$, a second fundamental form $\II_t$ and a mean curvature $H_t$ on $\partial N$. In this situation, the variations $\delta \I$, $\delta \II$ and $\delta H$ have to be understood as the derivatives in $t$ of the families of tensors $(\I_t)_t$, $(\II_t)_t$ and $(H_t)_t$ on the fixed manifold $\partial N$. The other possibility is to \emph{fix} a hyperbolic metric $g$ on a manifold $N'$, and to look at a $1$-\hsk parameter family of subsets $N_t$ of $N'$, whose boundaries are smooth and vary regularly in $t$, The relations that we are going to describe hold in both cases, so we will intentionally give "ambiguous" statements that can be applied in both the situations.

\begin{theorem}[{Differential Schl\"afli formula, \cite[Theorem~1, Theorem~2]{schlenker_rivin1999schlafli}}] \label{thm:diff_schlafli_formula}
	The variation of the hyperbolic volume of $N$ can be expressed as follows:
	\[
	\delta \Vol(N) = \frac{1}{2} \int_{\partial N} \left( \delta H + \frac{1}{2} \scall{\delta \I}{\II} \right) \dd{a} .
	\]
\end{theorem}

The $W$-\hsk volume of a smooth convex subset $N$ of a hyperbolic manifold is defined by the following expression:
\[
W(N) \defin \Vol(N) - \frac{1}{4} \int_{\partial N} H \dd{a} .
\]
Similarly, we define the dual volume of $N$ to be:
\[
V^*(N) \defin \Vol(N) - \frac{1}{2} \int_{\partial N} H \dd{a} .
\]
The first relation expresses the vatiation of the $W$-\hsk volume:

\begin{proposition} \label{prop:variation_W_volume}
	\[
	\delta W(N) = \frac{1}{4} \int_{\partial N} \left( \scall*{\delta \II}{\III - \frac{H}{2}\II}_{\II} + \frac{\delta K^e}{2 K^e} H \right) \dd{a} .
	\]	
	\begin{proof}
		Since $\I = \II \ B^{-1}$, we have $\delta \I = \delta \II \ B^{-1} - \II \ B^{-1} \ \delta B \ B^{-1}$. Therefore
		\begin{align*}
		\scall{\delta \I}{\II} & = \tr(\I^{-1} \ \delta \I \ \I^{-1} \ \II) \\
		& = \tr(B \ \II^{-1} (\delta \II \ B^{-1} - \II \ B^{-1} \ \delta B \ B^{-1}) B \ \II^{-1} \ \II) \\
		& = \tr(B \ \II^{-1} \ \delta \II) - \tr(\delta B) \\
		& = \tr(\II^{-1} \ \III \ \II^{-1} \ \delta \II) - \delta H \\
		& = \scall{\III}{\delta \II}_{\II} - \delta H .
		\end{align*}
		Using the fact that $\dd{a}_g = \sqrt{ \det(g_{i j}) } \dd{x}^1 \wedge \dd{x}^2$, we find $\delta (\dd{a}_g) = \frac{1}{2} \scall{\delta g}{g}_g \dd{a}_g $. Hence we have:
		\begin{align*}
		\delta \left( \dd{a} \right) & = \delta \left( \frac{\dd{a}_{\II}}{\sqrt{K^e}} \right) \\
		& = - \frac{\delta K^e}{2 (K^e)^{3/2}} \dd{a}_{\II} + \frac{1}{2 \sqrt{K^e}} \scall{\delta \II}{\II}_{\II} \dd{a}_{\II} \\
		& = \left( - \frac{\delta K^e}{2 K^e} + \frac{1}{2} \scall{\delta \II}{\II}_{\II} \right) \dd{a} .
		\end{align*}
		Applying Theorem \ref{thm:diff_schlafli_formula} and using the relations found above, we obtain:
		\begin{align*}
		\delta W(N) & = \delta \Vol(N) - \frac{1}{4} \delta \left( \int_{\partial N} H \dd{a} \right) \\
		& = \frac{1}{2} \int_{\partial N} \left( \delta H + \frac{1}{2} \scall{\delta \I}{\II} \right) \dd{a} - \frac{1}{4} \int_{\partial N} \left( \delta H \dd{a} + H \delta(\dd{a}) \right) \\
		& = \frac{1}{4} \int_{\partial N} \left( 2 \delta H + \scall{\III}{\delta \II}_{\II} - \delta H - \delta H - H  \left( - \frac{\delta K^e}{2 K^e} + \frac{1}{2} \scall{\delta \II}{\II}_{\II} \right) \right) \dd{a} \\
		& = \frac{1}{4} \int_{\partial N} \left( \scall*{\delta \II}{\III - \frac{H}{2}\II}_{\II} + \frac{\delta K^e}{2 K^e} H \right) \dd{a} ,
		\end{align*}
		which proves the statement.
	\end{proof}
\end{proposition}

On the other hand, the variation formula of the dual volume satisfies:

\begin{proposition} \label{prop:variation_dual_volume}
	\[
	\delta V^*(N) = \frac{1}{4} \int_{\partial N} \scall{\delta \I}{\II - H \I} \dd{a}
	\]
	\begin{proof}
		We need to apply the differential Schl\"afli formula and to express the variation of the term $\int H \dd{a}$. We can proceed similarly to what done for the $W$-\hsk volume, obtaining
		\begin{align*}
			\delta V^*(N) & = \delta \Vol(N) - \frac{1}{2} \delta \left( \int_{\partial N} H \dd{a} \right) \\
			& = \frac{1}{2} \int_{\partial N} \left( \delta H + \frac{1}{2} \scall{\delta \I}{\II} \right) \dd{a} - \frac{1}{2} \int_{\partial N} \left( \delta H \dd{a} + H \delta(\dd{a}) \right) \\
			& = \frac{1}{4}	\int_{\partial N} \left( 2 \delta H + \scall{\delta \I}{\II} - 2 \delta H - 2 \ \frac{H}{2} \scall{\delta \I}{\I} \right) \dd{a} \\
			& = \frac{1}{4}	\int_{\partial N} \scall{\delta \I}{\II - H \I} \dd{a} ,
		\end{align*}
		where, in the second step we used the fact that $\delta(\dd a) = \frac{1}{2} \scall{\delta \I}{\II} \dd{a}$. This concludes the proof of the statement.
	\end{proof}
\end{proposition}



\emergencystretch=1em

\printbibliography
	
\end{document}